\documentclass[11pt,a4paper]{article}

\usepackage{latexsym,amsfonts,amsmath,amsthm,graphics}


\usepackage{graphicx}       
\usepackage[latin1]{inputenc}
\usepackage[T1]{fontenc}
\usepackage{microtype} 

\usepackage{cite}
\usepackage{array,colortbl}
\usepackage{amsmath,amsfonts,amssymb,bm,amsthm,mathtools,mathrsfs}
\DeclareFontFamily{U}{mathx}{\hyphenchar\font45}
\DeclareFontShape{U}{mathx}{m}{n}{
	<5> <6> <7> <8> <9> <10>
	<10.95> <12> <14.4> <17.28> <20.74> <24.88>
	mathx10
}{}
\DeclareSymbolFont{mathx}{U}{mathx}{m}{n}
\DeclareFontSubstitution{U}{mathx}{m}{n}
\DeclareMathAccent{\widecheck}{0}{mathx}{"71}

\def\citep#1#2{\cite[{#1}]{#2}}

\usepackage{algorithm}
\usepackage{algorithmic}

\newcommand{\C}{{\mathbb{C}}} 
\newcommand{\F}{{\mathbb{F}}} 
\newcommand{\N}{{\mathbb{N}}} 
\newcommand{\R}{{\mathbb{R}}} 
\newcommand{\Z}{{\mathbb{Z}}} 
\newcommand{\FF}{{\mathbb{F}}} 

\DeclareSymbolFont{bbold}{U}{bbold}{m}{n}
\DeclareSymbolFontAlphabet{\mathbbold}{bbold}

\newcommand{\bsg}{{\boldsymbol{g}}}

\newcommand{\bsk}{{\boldsymbol{k}}}

\newcommand{\bsu}{{\boldsymbol{u}}}
\newcommand{\bsv}{{\boldsymbol{v}}}

\newcommand{\bsx}{{\boldsymbol{x}}}

\newcommand{\bszero}{{\boldsymbol{0}}} 
\newcommand{\bsone}{{\boldsymbol{1}}}  

\newcommand{\bsgamma}{{\boldsymbol{\gamma}}}

\newcommand{\bseta}{{\boldsymbol{\eta}}}

\newcommand{\calD}{{\mathcal{D}}}

\newcommand{\calO}{{\mathcal{O}}}

\newcommand{\setu}{{\mathfrak{u}}}

\newcommand{\mcol}{{\mathpunct{:}}}
\newcommand{\wal}{{\rm wal}}
\newcommand{\icomp}{\mathrm{i}}
	
\newcommand{\abs}[1]{\left\vert#1\right\vert}
\newcommand{\norm}[1]{\left\Vert#1\right\Vert}
\newcommand{\floor}[1]{\left\lfloor #1 \right\rfloor} 
\newcommand{\rd}{\,\mathrm{d}} 

\providecommand{\argmin}{\operatorname*{argmin}}

\newcommand{\supp}{\operatorname{supp}}
\newcommand{\tr}{{\rm tr}}
\newcommand{\tpmod}[1]{{\;(\operatorname{mod}\;#1)}}

\usepackage[hidelinks,hypertexnames=false,hyperindex=true,pdfpagelabels,linktoc=all]{hyperref}
\usepackage{bookmark}
\makeatletter 
\providecommand*{\toclevel@author}{999}
\providecommand*{\toclevel@title}{0}
\makeatother

\usepackage{enumitem}

\theoremstyle{plain}
\newtheorem{theorem}{Theorem}
\newtheorem{proposition}{Proposition}
\newtheorem{lemma}{Lemma}

\theoremstyle{definition}
\newtheorem{definition}{Definition}

\theoremstyle{remark}
\newtheorem{remark}{Remark}

\usepackage[T1]{fontenc}
\usepackage{lmodern}
\usepackage[a4paper]{geometry}
\usepackage{inputenc}
\usepackage{amsmath,amsfonts,amssymb}
\usepackage{bm}
\usepackage{mathtools}
\usepackage{hyperref}
\usepackage{tikz}
\usepackage{caption}
\usepackage{listings}
\usepackage{float}
\usepackage{pgfplots}
\usepackage{subcaption}
\usepackage{enumitem}
\usepackage{xcolor}

\usepackage{booktabs}
\usepackage{multirow}

\allowdisplaybreaks

\usepackage{changebar}

\pgfplotsset{every tick label/.append style={font=\scriptsize}}
\newenvironment{customlegend}[1][]{
	\begingroup
	\csname pgfplots@init@cleared@structures\endcsname
	\pgfplotsset{#1}
}{
	\csname pgfplots@createlegend\endcsname
	\endgroup
}

\def\addlegendimage{\csname pgfplots@addlegendimage\endcsname}
\pgfplotsset{
	cycle list={%
		{draw=black,mark=star,solid},
		{draw=black, mark=square,solid},
		{draw=black,mark=+,solid},
		{black,mark=o},
		{draw=black, mark=none,solid}}
}

\definecolor{mycolor-alpha1}{rgb}{0.30000,0.45000,0.85000}
\definecolor{mycolor-alpha2}{rgb}{0.05000,0.65000,0.20000}
\definecolor{mycolor-alpha3}{rgb}{0.64000,0.22500,0.75000}

\definecolor{mycolor1}{rgb}{0.48500,0.70000,1.00000}
\definecolor{mycolor2}{rgb}{0.00000,0.20000,0.50000}
\definecolor{mycolor3}{rgb}{0.95000,0.15000,0.05000}
\definecolor{mycolor4}{rgb}{0.10000,0.80000,0.15000}
\definecolor{mycolor5}{rgb}{0.00000,0.50000,0.25000}
\definecolor{mycolor6}{rgb}{0.84000,0.29000,1.00000}
\definecolor{mycolor7}{rgb}{0.44000,0.16500,0.50000}

\definecolor{mycolor1-time}{rgb}{0.00000,0.44700,0.74100}%
\definecolor{mycolor2-time}{rgb}{0.85000,0.32500,0.09800}%
\definecolor{mycolor3-time}{rgb}{0.49400,0.18400,0.55600}%
\definecolor{mycolor4-time}{rgb}{0.92900,0.69400,0.12500}%

\begin{document}

\title{Component-by-component digit-by-digit construction of good polynomial lattice rules in weighted Walsh spaces}

\author{Adrian Ebert\thanks{A.~Ebert, P.~Kritzer, and O.~Osisiogu are supported by the Austrian Science Fund (FWF): Project F5506, 
which is part of the Special Research Program ``Quasi-Monte Carlo Methods: Theory and Applications''.}, Peter Kritzer, 
Onyekachi Osisiogu, Tetiana Stepaniuk\thanks{T.~Stepaniuk is supported by the Alexander von Humboldt Foundation.}}

\date{\today}

\maketitle

\begin{abstract}
	We consider the efficient construction of polynomial lattice rules, which are special cases of so-called quasi-Monte Carlo (QMC) rules.
	These are of particular interest for the approximate computation of multivariate integrals where the dimension $d$ may be in the hundreds or thousands. 
	We study a construction method that assembles the generating vector, which is in this case a vector of polynomials over a finite field, of the polynomial lattice 
	rule in a digit-by-digit (or, equivalently, coefficient-by-coefficient) fashion. As we will show, the integration error of the corresponding QMC rules achieves 
	excellent convergence order, and, under suitable conditions, we can vanquish the curse of dimensionality by considering function spaces equipped with coordinate weights. The construction algorithm is based on a quality measure that is independent of the underlying smoothness of the function space and can be implemented 
	in a fast manner (without the use of fast Fourier transformations). Furthermore, we illustrate our findings with extensive numerical results.
\end{abstract}



\section{Introduction} \label{sec:intro}

In this article we study the problem of multivariate numerical integration for a subclass of square-integrable functions $f \in L^2([0,1]^d)$. We consider special 
instances of so-called quasi-Monte Carlo (QMC) rules, which are methods to approximate integrals 
$$
I_d(f) =\int_{[0,1]^d } f(\bsx) \rd \bsx
$$
by equal-weight quadrature rules, 
$$
Q_{N,d}(f)=\frac{1}{N} \sum_{n=0}^{N-1} f(\bsx_n),
$$
where the integration nodes $\bsx_0,\bsx_1,\ldots,\bsx_{N-1}$ are deterministically chosen in $[0,1]^d$. This is in contrast 
to Monte Carlo rules, where the integration nodes are chosen randomly; with QMC rules, we try to make a deliberate and 
sophisticated choice of the points $\bsx_n$ with the aim of obtaining better error bounds than for Monte Carlo. 
The crucial challenge is to find integration nodes yielding a low approximation error simultaneously for a large class 
of functions that may depend on many variables. This means that, usually, one needs to be able to find millions 
of good integration nodes in very high dimensions which is a considerable computational challenge. 

In the literature on QMC methods, there are two main concepts that are commonly made use of when trying to find sets of integration nodes 
with good properties. These are, on the one hand, lattice point sets, as introduced independently by Korobov (see \cite{Kor59}) and Hlawka (see \cite{H62}). 
For more recent introductions to lattice rules, we refer to \cite{Nied92, SJ94}. The other class of commonly used QMC integration nodes 
is that of (digital) $(t,m,d)$-nets and $(t,d)$-sequences, as introduced by Niederreiter, building up on ideas by Sobol' and Faure (see \cite{N87,Nied92}).
A special case of $(t,m,d)$-nets, namely so-called polynomial lattice point sets, is the focus of the present paper. These point sets were introduced 
in \cite{N92a}, and have their name since their structure can be viewed as analogous to (ordinary) lattice point sets. 

While the construction principle of 
lattice point sets is based on integer arithmetic, polynomial lattice point sets are based on polynomial arithmetic over finite fields. To be more precise, 
we will fix a prime $b$, and consider the finite field $\F_b$ with $b$ elements. A polynomial lattice point set with $b^m$ points in $[0,1]^d$ is constructed by 
means of a modulus $p\in \F_b [x]$ with $\deg (p)=m$, and a generating vector $\bsg\in (\F_b [x])^d$ (we refer to Section \ref{sec:plr} for the precise definition). 
The QMC rule using the polynomial lattice point set as integration nodes is then called a polynomial lattice rule. It will be convenient in this paper to 
assume that the modulus has the form $p_m (x)=x^m$. However, it is crucial to note that not every choice of the generating vector $\bsg$ yields 
a polynomial lattice point set that has good properties, in the sense that the integration error of the corresponding polynomial lattice rule 
is sufficiently low. On the contrary, it is usually highly non-trivial to find good generating vectors of polynomial lattice rules, and there 
are (except for special cases) no explicit constructions of such good generating vectors known. Hence, one has to resort to computer search algorithms
for finding generating vectors of polynomial lattice point sets of high quality. Regarding the error measure, we consider in this paper 
the worst-case setting, i.e., we consider a particular normed function space and the supremum of the integration error over the unit ball 
of the space. 

It is known that (ordinary) lattice rules are well suited for the numerical integration of functions with pointwise convergent 
Fourier series (see again, e.g., \cite{Nied92} or \cite{SJ94}). On the other hand, polynomial lattice rules are usually applied for the 
numerical integration of functions that can be represented by Walsh series (cf. \cite{DKPS05,DP05,DP10}). We will therefore define a reproducing 
kernel Hilbert space based on Walsh functions in Section \ref{sec:walsh_space}, which will be considered throughout the paper. 
The function space under consideration will be characterized by a smoothness parameter $\alpha$ (in some publications this parameter is also 
referred to as ``digital smoothness parameter'' in the context of Walsh series). Indeed, the parameter $\alpha$ is linked to the speed 
of decay of the Walsh coefficients of the functions in our space, but there is also a connection to the number of derivatives that exist 
for the elements of the space (we refer to \cite{DP10} and the references therein for details). 

The function space considered here is closely related to other function spaces considered in the literature, such as in \cite{DKPS05, DP05, DP10}; indeed, results that 
we show for the space considered in the present paper immediately imply corresponding results for some of the Walsh spaces 
considered in these references. Furthermore, our Hilbert space will be a ``weighted'' function space in the sense of Sloan and Wo\'{z}niakowski (cf. \cite{SW98}). 
This means that we assign non-negative real numbers (weights) to the coordinates, or groups of coordinates, of the integration problem, in order to model the different 
influence of the coordinates on the problem. As pointed out in \cite{SW98} and numerous other papers, this method is justified by 
practical high-dimensional problems in which different coordinates may indeed have a very different degree of influence on the value of an integral. 
The weights will be incorporated in the inner product and norm of the function space in a suitable way. Using this setting, it is plausible that a nominally very 
high-dimensional problem may have a rather low ``effective dimension'', i.e., only a certain, possibly small, part of the components has a significant influence on the integration problem and the error made by approximative algorithms. This may then yield situations where a curse of dimensionality can be avoided. 

In the present paper, we will restrict ourselves, for technical reasons, to considering the most common choice of weights, so-called 
product weights, but we suspect that the construction of QMC rules presented here could also work for other choices of weights. We 
refer to Section \ref{sec:algorithm} for further comments on this question.

The first efficient construction of good generating vectors of polynomial lattice point sets was done in \cite{DKPS05}. In that paper, 
the authors considered the so-called component-by-component (CBC) approach, which is a greedy algorithm to construct one component of the 
generating vector at a time. CBC algorithms were first considered for ordinary lattice point sets, with the first examples in the literature 
going back to Korobov (cf. \cite{Kor63}), and later a rediscovery by Sloan and Reztsov (cf. \cite{SR02}). The fast CBC construction, which 
is due to Cools and Nuyens (see, e.g., \cite{N14, NC06a, NC06b}), makes the CBC construction computationally competitive and is currently 
the standard method to construct high-dimensional lattice point sets of good quality.It is well known (see, e.g., \cite{DKPS05} and again \cite{N14}) 
that CBC constructions also work for the efficient search for generating vectors of polynomial lattice point sets; and also in this case, a fast algorithm is available. 

In the present paper, we present another, different algorithm to construct generating vectors of polynomial lattice point sets in an efficient 
way. This construction is also based on a component-by-component approach. However, as opposed to the CBC algorithms for polynomial lattice point 
sets currently available in the literature, our new approach constructs the single components of the generating vector $\bsg$ ``digit-by-digit'' 
and the used search criterion is independent of the smoothness parameter $\alpha$.
Actually, the term ``digit-by-digit'' is based on a similar approach that exists for ordinary lattice point sets (see \cite{Kor82, Kor82Eng}, 
and for similar results in a more up-to-date setting, \cite{EKNO2020}). In the context of polynomial lattice point sets, the generating 
vector $\bsg$ consists of polynomials, so it would  be more appropriate to speak of a ``coefficient-by-coefficient'' instead of a 
``digit-by-digit'' construction. However, to stay consistent regarding the name of the method, and to avoid confusion with the ``component-by-component''
approach, we keep the name ``digit-by-digit'' construction also for polynomial lattice rules. In fact, the algorithm which we will present in Section 
\ref{sec:algorithm} contains two loops. An outer loop in which the different components are constructed, and an inner loop in which 
the coefficients (digits) of each component of the generating vector are constructed. Both loops can be regarded as greedy, i.e., choices that 
have been made in previous steps are kept fixed. 

We will show that the polynomial lattice rules obtained by our new construction method satisfy upper error bounds that are arbitrarily close 
to the optimal convergence rate. Furthermore, under suitable conditions on the coordinate weights, we can vanquish the curse of dimensionality, 
i.e., avoid exponential dependence of the error on the dimension $d$ of the integration problem, or even obtain error bounds that are independent 
of the dimension. 

The rest of the paper is structured as follows. In Section \ref{sec:walsh-polylat}, we introduce the function space setting as well as polynomial lattice rules, and analyze 
the corresponding worst-case error expression. In Section \ref{sec:cbc-dbd}, we derive the component-by-component digit-by-digit (or, for short, CBC-DBD)
construction algorithm for polynomial lattice rules and study the worst-case error behavior of the resulting integration rules. 
In Section \ref{sec:fast_impl}, we show that the introduced construction method can be implemented in a fast manner, competitive with state-of-the-art construction algorithms. Finally, the article is concluded in Section \ref{sec:num}, where we illustrate our main results by numerical experiments.

To conclude this introductory section, we fix some notation. In what follows, we denote the set of positive integers by $\N$ and the set of 
non-negative integers by $\N_0$. To denote subsets of components, we use fraktur font, e.g., $\setu \subset \N$ and 
additionally write shorthand $1 \mcol d\} := \{1,\ldots,d\}$. For the projection of a vector $\bsx \in [0,1]^d$ or $\bsk \in \N^d$ 
onto the components in a set $\setu \subseteq \{1 \mcol d\}$ we write $\bsx_\setu = (x_j)_{j \in \setu}$ or $\bsk_\setu = (k_j)_{j \in \setu}$, respectively. 
With a slight abuse of notation, we will frequently identify elements of the finite field $\F_b$ of prime cardinality $b$ with elements of the group of integers modulo $b$ denoted by $\Z_b$.

\section{Polynomial lattice rules in weighted Walsh spaces} \label{sec:walsh-polylat}

In this article we consider numerical integration of a sub-class of the square-integrable functions 
$f \in L^2([0,1]^d)$ which can be represented in terms of their Walsh series. 
This particular series representation of a function is based on the so-called Walsh functions, which are defined as follows.

\begin{definition} \label{def:walsh_functions}
	Let $b \ge 2$ be an integer. For a non-negative integer $k$, we define the $k$-th Walsh function $_b\wal_k :[0,1) \to \C$ by
	\begin{equation*}
		_b\wal_k (x)
		:=
		\mathrm{e}^{2\pi\icomp (\kappa_0 \xi_1 + \kappa_1 \xi_2 + \cdots + \kappa_{a-1} \xi_{a})/b}
	\end{equation*}
	with $x\in [0,1)$ and base $b$ representations $k=\kappa_0 + \kappa_1 b + \cdots \kappa_{a-1} b^{a-1}$ and 
	$x=\xi_1 b^{-1} + \xi_2 b^{-2} + \cdots$ (unique in the sense that infinitely many of the $\xi_i$ must be different from $b-1$) 
	with coefficients $\kappa_i, \xi_i \in \{0,1,\ldots,b-1\}$.

	For $d \in \N$, an integer vector $\bsk=(k_1,\ldots,k_d) \in \N_0^d$ and $\bsx=(x_1,\ldots,x_d)\in [0,1)^d$, we define 
	the $\bsk$-th ($d$-variate) Walsh function $_b\wal_{\bsk} :[0,1)^d \to \C$ by
	\begin{equation*}
		_b\wal_{\bsk} (\bsx)
		:=
		\prod_{j=1}^d \ _b\wal_{k_j} (x_j)
		.
	\end{equation*}
\end{definition}

In the following, we will consider the base $b \ge 2$ as fixed (for the sake of simplicity, we will assume that $b$ is prime), and then simply write $\wal_k$ or $\wal_{\bsk}$ 
instead of $_b\wal_k$ or $_b\wal_{\bsk}$, respectively. It is known (see, e.g., \cite{DP10}) that the Walsh functions 
in any fixed base $b$ form an orthonormal basis of $L^2 ([0,1]^d)$. 

As indicated, we consider a class of square-integrable functions that can be represented in terms of their Walsh series, that is,
\begin{equation} \label{eq:Walsh_series}
	f(\bsx)
	=
	\sum_{\bsk \in \N_0^d} \hat{f}(\bsk) \, \wal_\bsk(\bsx)
	\quad \text{with} \quad
	\hat{f}(\bsk)
	:=
	\int_{[0,1]^d} f(\bsx) \, \overline{\wal_\bsk(\bsx)} \rd \bsx
	,
\end{equation}
where we call $\hat{f}(\bsk)$ the $\bsk$-th Walsh coefficient of $f$.

It is known from the literature on QMC methods in the past decades that it is advantageous to choose the integration nodes 
of a QMC rule such that there exists an efficient way of expressing the integration error for elements in the function class under consideration. 
In the case where the integrand $f$ can be represented in terms of Walsh series as in \eqref{eq:Walsh_series}, it is common to consider quasi-Monte Carlo 
rules which are based on so-called digital nets and sequences. Digital $(t,m,d)$-nets are point sets consisting of $b^m$ elements in $[0,1]^d$ that satisfy certain regular distribution properties, and were in their most general form introduced in \cite{N92a} (see also \cite{Nied92}). These point sets are generated by using $d$ generating matrices $C_1,\ldots,C_d$ over a finite field or ring. In particular, for a digital $(t,m,d)$-net $P=\{\bsx_0,\ldots,\bsx_{b^m-1}\} \subset [0,1]^d$ constructed over $\Z_b=\{0,1,\ldots,b-1\}$ with generating matrices $C_1,\ldots,C_d \in \Z_b^{m \times m}$ the integration error of a QMC rule based on $P$ takes a special form. It is commonly known, see, e.g., \cite[Theorem 6.4]{DKS13}, that approximating the integral $I_d(f)$ of a $d$-variate function $f$ using a QMC rule $Q_{b^m,d}(f;P)$, that is,
\begin{equation*}
	Q_{b^m,d}(f)=Q_{b^m,d}(f;P)
	:=
	\frac{1}{b^m} \sum_{n=0}^{b^m-1} f(\bsx_n)
	\approx
	\int_{[0,1]^d} f(\bsx) \rd \bsx 
	=:
	I_d(f)
	,
\end{equation*}
leads to an integration error of the form
\begin{equation} \label{eq:int_error}
	Q_{b^m,d}(f;P) - I_d(f)
	=
	\sum_{\bszero \ne \bsk \in \calD} \hat{f}(\bsk)
\end{equation}
with the dual net $\calD = \calD(C_1,\ldots,C_d) := \{\bsk \in \N_0^d \mid C_1^\top \widetilde{\tr}_m(\vec{k}_1) + \dots + C_d^\top \widetilde{\tr}_m(\vec{k}_d) = \overline{\bszero} \}$, where for $k \in \N_0$ with base $b$ expansion $k = \kappa_0 + \kappa_1 b + \dots + \kappa_a b^a$ 
we define the vector $\widetilde{\tr}_m(\vec{k})=(\kappa_0,\kappa_1, \ldots, \kappa_{m-1}) \in \Z_b^m$, and where we denote by $\overline{\bszero}$ the zero vector in $\Z_b^m$. Equation \eqref{eq:int_error} is a consequence of the following character property of Walsh functions,
\begin{equation*} 
	\frac{1}{b^m}\sum\limits_{n=0}^{b^m-1} \mathrm{wal}_{\bsk} (\bsx_n)
	=
	\left\{\begin{array}{cc}
		1, & {\text{if }}  C_1^\top \widetilde{\tr}_m(\vec{k}_1) + \dots + C_d^\top \widetilde{\tr}_m(\vec{k}_{d}) = \overline{\bszero} , \\ 
		0, & {\text{otherwise.}}
	\end{array}\right.
\end{equation*}
We will also use this property in the subsequent analysis.

\subsection{The weighted Walsh space}\label{sec:walsh_space}

Based on the decay of the Walsh coefficients $\hat{f}(\bsk)$ in \eqref{eq:Walsh_series} we will define a function space for the integrands considered in this paper. 
As mentioned in the introduction, this space will be equipped with weights to model the varying influence of the coordinates. 
To this end, let $\bsgamma =(\gamma_j)_{j\ge 1}$ be a non-increasing sequence of positive real numbers. 
The weights $\gamma_j$ will appear in the definition of the inner product and norm of the function space 
defined below. Intuitively, we can think of the weight $\gamma_j$ describing the degree of influence of the $j$-th variable on the integration problem. 
Hence, we assume (w.l.o.g.) that the coordinates are ordered according to their influence. It will also be convenient to define 
\begin{equation*}
	\gamma_\setu
	:=
	\prod_{j\in\setu} \gamma_j
\end{equation*}
for a subset $\setu \subseteq \{1 \mcol d\}$, and to additionally set $\gamma_{\emptyset}$ to equal $1$. The weights $\gamma_\setu$ 
are (for obvious reasons) called product weights. In the recent literature on QMC rules, also other types of weights have been considered, but we will restrict ourselves to product weights here. We refer to \cite{DKS13} for further information on this subject. 

For prime base $b \ge 2$ and given smoothness parameter $\alpha > 1$, we set $\psi_b(k):=\floor{\log_b(k)}$ for $k \in \N$ and define the decay function $r_\alpha: \N_0 \to \R$ by
\begin{equation*}
	r_\alpha (k)=r_\alpha (b,k)
	:=
	\left\{\begin{array}{cc}
	1, & {\text{if }} k=0 , \\ 
	b^{\alpha \psi_b (k)}, & {\text{if }} k\ne0 ,
	\end{array}\right.
\end{equation*}
with $k\in\N_0$. It is also convenient to define the quantity
\begin{equation*}
	\mu_b (\alpha)
	:=
	\sum_{k=1}^\infty (r_{\alpha} (k))^{-1}= \sum_{a=0}^\infty \frac{1}{b^{a\alpha}} \sum_{k=b^a}^{b^{a+1}-1} 1=\sum_{a=0}^\infty \frac{(b-1)b^a}{b^{a\alpha}}
	=
	\frac{b^{\alpha}(b-1)}{b^\alpha - b}
	.
\end{equation*} 

For the multivariate case with dimension $d\in\N$, integer vector $\bsk=(k_1,\ldots,k_d) \in \N_0^d$, and a sequence of weights 
$\bsgamma=(\gamma_j)_{j\ge 1}$, we define the weighted decay functions
\begin{equation*}
	r_\alpha (\bsk)
	:=
	\prod_{j=1}^d r_\alpha (k_j) 
	\quad \text{and} \quad
	r_{\alpha,\bsgamma}(\bsk)
	:=
	\gamma_{\supp(\bsk)}^{-1} \, r_\alpha (\bsk) 
	=
	\gamma_{\supp(\bsk)}^{-1} \prod_{j \in \supp(\bsk)} b^{\alpha \psi_b(k_j)}
\end{equation*}
with $\supp(\bsk) := \{ j \in \{1 \mcol d\} \mid k_j \ne 0 \}$.

Using this decay function, we can  estimate the integration error obtained in \eqref{eq:int_error} by
\begin{align} \label{eq:Hoelder}
	\abs{Q_{b^m,d}(f;P) - I_d(f)}
	&=
	\abs{\sum_{\bszero \ne \bsk \in \calD} \hat{f}(\bsk)}
	=
	\abs{\sum_{\bszero \ne \bsk \in \N_0^d} \hat{f}(\bsk) \, r_{\alpha,\bsgamma}(\bsk) \, (r_{\alpha,\bsgamma}(\bsk))^{-1} \, \bsone_{\calD}(\bsk)} \nonumber \\
	&\le
	\left( \sup_{\bsk \in \N_0^d} |\hat{f}(\bsk)| \, r_{\alpha,\bsgamma}(\bsk) \right) \left( \sum_{\bszero \ne \bsk \in \calD} (r_{\alpha,\bsgamma}(\bsk))^{-1} \right)
\end{align}
with $\bsone_{\calD}$ denoting the indicator function of the dual lattice $\calD$. Based on this estimate, we define, for real
$\alpha>1$ and a sequence of strictly positive weights $\bsgamma = (\gamma_j)_{j\ge 1}$, the weighted Walsh space as
\begin{equation*}
	W_{d,\bsgamma}^{\alpha}
	:=
	\{f \in L^2([0,1]^d) \mid \norm{f}_{W_{d,\bsgamma}^{\alpha}} < \infty \}
\end{equation*}
with corresponding norm $\norm{\cdot}_{W_{d,\bsgamma}^{\alpha}}$ given by
\begin{equation} \label{eq:norm}
	\norm{f}_{W_{d,\bsgamma}^{\alpha}}
	:=
	\sup_{\bsk \in \N_0^d} |\hat{f}(\bsk)| \, r_{\alpha,\bsgamma}(\bsk)
	.
\end{equation}

\begin{remark}
	 We remark that the definition of the norm implies that functions in $W_{d,\bsgamma}^{\alpha}$ have an absolutely 
	 convergent Walsh series which converges pointwise (see, e.g., \cite{DP10}). 
\end{remark}

\begin{remark}
	 We would like to note here that in many recent papers (e.g., \cite{DKPS05,DP05}), a slightly different function space 
	 $\widetilde{W}_{d,\bsgamma}^{\alpha}$
	 based on Walsh functions has been studied. In $\widetilde{W}_{d,\bsgamma}^{\alpha}$ the norm is not given as an $\infty$-norm 
	 as in \eqref{eq:norm}, but in the $L_2$-sense, i.e.,
	 \begin{equation*}
		\norm{f}_{\widetilde{W}_{d,\bsgamma}^{\alpha}}
		:=
		\sum_{\bsk \in \N_0^d} |\hat{f}(\bsk)|^2 \, r_{\alpha,\bsgamma}(\bsk)
		.
	\end{equation*}
	This definition of the norm corresponds to alternatively applying H\"{o}lder's inequality with $p=q=2$ in the bound on the integration error that led to \eqref{eq:Hoelder}.
	As we will see below, the worst-case error expressions for $W_{d,\bsgamma}^{\alpha}$ and $\widetilde{W}_{d,\bsgamma}^{\alpha}$ are closely related to each other. 
\end{remark}

In order to assess the quality of the QMC methods constructed later on, we will use the worst-case error in the weighted Walsh space as the error criterion. Indeed, the worst-case error for the QMC rule $Q_{b^m,d}(\cdot;P)$ in the space $W_{d,\bsgamma}^{\alpha}$ is defined as
\begin{equation*}
	e_{b^m,d,\alpha,\bsgamma}(P)
	:=
	\sup_{\substack{f \in W_{d,\bsgamma}^{\alpha} \\ \|f\|_{W_{d,\bsgamma}^{\alpha}} \le 1}} | I_d(f) - Q_{b^m,d}(f;P)|
	.
\end{equation*}
A useful formula for the worst-case error for $(t,m,d)$-nets in the function space $W_{d,\bsgamma}^{\alpha}$ is given in the following theorem. 

\begin{theorem}\label{thm:wce_dig_net}
	Let $m,d \in \N$, $\alpha > 1$, $b \ge 2$, and a sequence of positive weights $\bsgamma = (\gamma_j)_{j\ge 1}$ be given. 
	Then the worst-case error $e_{b^m,d,\alpha,\bsgamma}(P)$ of the QMC rule $Q_{b^m,d}(\cdot;P)$ based on the digital 
	$(t,m,d)$-net $P=\{\bsx_0,\ldots,\bsx_{b^m-1}\}$ with generating matrices $C_1,\ldots,C_d$ in the space $W_{d,\bsgamma}^{\alpha}$ satisfies
	\begin{equation}\label{eq:wce-infty}
		e_{b^m,d,\alpha,\bsgamma}(P)
		=
		\sum_{\bszero \ne \bsk \in \calD} (r_{\alpha,\bsgamma}(\bsk))^{-1}
		.
	\end{equation}
\end{theorem}

\begin{proof}
	Recalling the definition of the worst-case error of the QMC rule $Q_{b^m,d}(\cdot;P)$,
	the combination of \eqref{eq:Hoelder} and the definition of $\|\cdot\|_{W_{d,\bsgamma}^{\alpha}}$ leads to the estimate
	\begin{equation*}
		e_{b^m,d,\alpha,\bsgamma}(P)
		\le
		\sup_{\substack{f \in W_{d,\bsgamma}^{\alpha} \\ \|f\|_{W_{d,\bsgamma}^{\alpha}} \le 1}} \|f\|_{W_{d,\bsgamma}^{\alpha}} 
		\sum_{\bszero \ne \bsk \in \calD} (r_{\alpha,\bsgamma}(\bsk))^{-1}
		\le
		\sum_{\bszero \ne \bsk \in \calD} (r_{\alpha,\bsgamma}(\bsk))^{-1} 
		.
	\end{equation*}
	Observing that the function $f_0$ with Walsh coefficients $\hat{f_0}(\bsk) =(r_{\alpha,\bsgamma}(\bsk))^{-1}$ has norm $\|f_0\|_{W_{d,\bsgamma}^{\alpha}} = 1$ and that its integration error equals
	\begin{equation*}
		Q_{b^m,d}(f_0,P) - I_d(f_0) 
		=
		\sum_{\bszero \ne \bsk \in \calD} (r_{\alpha,\bsgamma}(\bsk))^{-1} ,
	\end{equation*}
	we obtain that the previous upper bound is attained such that the claimed identity follows. 
\end{proof}

\begin{remark}
	 Returning to the alternative Walsh space $\widetilde{W}_{d,\bsgamma}^{\alpha}$ once again, 
	 it is known from \cite{DP05} that the worst-case error in this space equals
	 \begin{equation*}
	    \left(\sum_{\bszero \ne \bsk \in \calD} (r_{\alpha,\bsgamma}(\bsk))^{-1}\right)^{1/2}
	    ,
	 \end{equation*}
	 which is just the square root of the worst-case error in $W_{d,\bsgamma}^{\alpha}$, as outlined in Theorem \ref{thm:wce_dig_net}. 
	 Therefore, we see that the worst-case errors in these Walsh spaces are intimately related to each other, and all results 
	 shown here for $W_{d,\bsgamma}^{\alpha}$ immediately yield corresponding results for $\widetilde{W}_{d,\bsgamma}^{\alpha}$.
\end{remark}

\subsection{Polynomial lattice rules}\label{sec:plr}

While Theorem \ref{thm:wce_dig_net} is a very useful result, the question of how to find and construct $(t,m,d)$-nets 
with a low integration error for practical purposes remains. One of the most powerful ways of obtaining nets is to consider a special case, 
namely so-called polynomial lattice point sets, as introduced by Niederreiter in \cite{N92a}. The name ``polynomial lattice point sets'' is 
due to the fact that the structure of polynomial lattice point sets is similar to that of ordinary lattice point sets as introduced by Korobov 
\cite{Kor59} and Hlawka \cite{H62}. However, while lattice point sets are based on integer arithmetic, polynomial lattice point sets are 
obtained by using polynomial arithmetic over finite fields. We also point out that there are nowadays variants of polynomial lattice point 
sets which are especially suited for integrating functions with higher smoothness (see, e.g., \cite{DP10}). However, we will not consider 
higher order polynomial lattices here, but restrict ourselves to the more classical construction scheme. 
We point out that polynomial lattice point sets are actually a special case of so-called digital $(t,m,d)$-nets, which 
can be constructed using generating matrices $C_1,\ldots,C_d$ over a finite field. 
For our purposes, though, it is more convenient to define these point sets in an alternative way. Before we give the precise definition, we need to introduce some notation. 

Let $\F_b ((x^{-1}))$ be the field of formal Laurent series over $\F_b$ with elements of the form
\begin{align*}
	L
	=
	\sum_{\ell=w}^\infty t_{\ell} x^{-\ell}
	,
\end{align*}
where $w$ is an arbitrary integer and all $t_{\ell}\in\F_b$. We further denote by $\F_b [x]$ the set of all polynomials over $\F_b$ 
and define the map $v_m: \F_b ((x^{-1}))\to [0,1)$ by
\begin{align*}
	v_m \left(\sum_{\ell =w}^\infty t_\ell \, x^{-\ell}\right)=\sum_{\ell=\max (1,w)}^m t_{\ell} \, b^{-\ell}
	.
\end{align*}
There is a close connection between the base $b$ expansions of natural numbers and the polynomial ring $\F_b [x]$. For $n\in \N_0$ with base $b$ expansion $n=n_0 + n_1 b + \cdots + n_a b^{a}$, we associate $n$ with the polynomial
\begin{align*}
	n(x)
	:=
	\sum_{k=0}^a n_k \, x^k \in \F_b [x].
\end{align*} 
The definition of a polynomial lattice point set is then given as follows. We note that here 
and in the following we consider the zero polynomial to have degree $-\infty$, hence the case $n=0$ is included 
in the following definition.

\begin{definition}[Polynomial lattice] \label{def:poly_lat}
	Let $b$ be prime and let $m,d \in \N$ be given. Furthermore, choose $p\in \F_b [x]$ with $\deg (p)=m$, and let 
	$g_1,\ldots,g_d \in \F_b [x]$. Then the point set $P(\bsg,p)$, defined as the collection of the $b^m$ points 
	\begin{equation*}
		\bsx_n
		:=
		\left(v_m \left( \frac{n(x)\, g_1(x)}{p(x)} \right),\ldots, v_m\left( \frac{n(x)\, g_d(x)}{p(x)} \right)\right) \in [0,1)^d
	\end{equation*}
	for $n \in \F_b[x]$ with $\deg(n)<m$, is called a polynomial lattice point set (we sometimes also refer to the point set 
	as polynomial lattice for short), where the vector $\bsg=(g_1,\ldots,g_d) \in (\F_b[x])^d$ is called the generating vector.
\end{definition}

As pointed out above, due to the construction principle and the similarities to the construction 
of (rank-1) lattices, $P(\bsg,p)$ is often called a (rank-1) polynomial lattice and a QMC rule using the point
set $P(\bsg,p)$ is referred to as a polynomial lattice rule (modulo $p$). 
Furthermore, note that one can restrict the choice of the components $g_j$ of $\bsg$ to the sets
\begin{align*}
	G_{b,m}
	:=
	\{ g\in\F_b [x] \mid \deg(g) < m \}
	\quad \text{or} \quad
	G^\ast_{b,m}
	:=
	\{ g\in\F_b [x] \setminus \{0\} \mid \deg(g) < m \}
	.
\end{align*}
We also add that it is known from the literature on polynomial lattice point sets that it is desirable to have $\gcd(g_j,p)=1$ for the components $g_j$ 
of $\bsg$, as this guarantees certain regularity properties. For prime $b$, the generating matrices $C_1,\ldots,C_d \in \F_b^{m \times m}$ of a polynomial lattice point set $P(\bsg,p)$ can be obtained from the generating vector $\bsg$ and $p$, cf. \cite[Theorem 10.5]{DP10}. It then follows that the dual net $\calD(\bsg,p)$ of a polynomial lattice with generating vector $\bsg$, modulus $p$ with $\deg(p)=m$, and generating matrices $C_1,\ldots,C_d$ equals (see, e.g., \cite[Lemma 4.40]{Nied92})
\begin{equation*}
	\calD(\bsg,p)
	=
	\{\bsk \in \N_0^d \mid C_1^\top 
	\widetilde{\tr}_m(\vec{k}_1) + \dots + C_d^\top \widetilde{\tr}_m(\vec{k}_d) = \overline{\bszero} \}
	=
	\{\bsk \in \N_0^d \mid \tr_m(\bsk) \cdot \bsg \equiv 0 \tpmod{p} \} 
	,
\end{equation*} 
where for two vectors $\bsu,\bsv \in (\FF_b[x])^d$ we define the vector dot product $\bsu \cdot \bsv = \sum_{j=1}^{d} u_j v_j$. 
Furthermore, for $k \in \N_0$ with $b$-adic expansion $k=\kappa_0 + \kappa_1 b + \cdots + \kappa_{a-1} b^{a-1}$, 
we define the truncation map $\tr_m: \N_0 \to G_{b,m}$ via
\begin{equation*}
	\tr_m(k)
	:=
	\kappa_0 + \kappa_1 x + \cdots +\kappa_{m-1} x^{m-1},
\end{equation*}
where we consider $\kappa_j$ as 0 if $j\ge a$. If we apply $\tr_m$ to a $d$-dimensional vector, we define its $d$-variate generalization $\tr_m(\bsk)$ to be applied 
componentwise. Furthermore, for a subset $\setu \subseteq \{1 \mcol d\}$ we introduce the notation
\begin{equation*}
	\calD_\setu
	=
	\calD_\setu(\bsg, p)
	=
	\calD_\setu(\bsg_\setu)
	:=
	\{ \bsk_\setu  \in \N^{|\setu|} \mid \tr_m(\bsk_\setu) \cdot \bsg_\setu \equiv 0 \tpmod{p} \} 
	.
\end{equation*}
Due to the obtained equivalence for the dual net of a polynomial lattice, the result in Theorem \ref{thm:wce_dig_net} 
also applies to polynomial lattice rules with $\calD(C_1,\ldots,C_d)$ replaced by $\calD (\bsg,p)$. Furthermore, we will henceforth denote the 
worst-case error of a QMC rule based on the polynomial lattice point set $P(\bsg,p)$ in the space $W_{d,\bsgamma}^\alpha$ by $e_{b^m,d,\alpha,\bsgamma}(\bsg)$.  

\subsection{The quality measure}

In this section we introduce an alternative quality measure which, opposed to the worst-case error expression $e_{b^m,d,\alpha,\bsgamma}$ in 
\eqref{eq:wce-infty} is independent of the parameter $\alpha$.

For $\alpha \ge 1$, given weight sequence $\bsgamma=(\gamma_j)_{j \ge 1}$, $m \in \N$, modulus $p \in \F_b[x]$ with $\deg(p)=m$, 
and $\bsg \in (\F_b[x])^d$, we define the quantities
\begin{equation} \label{eq:qual_meas}
	T_{\bsgamma}(\bsg,p)
	:= 
	\sum_{\bszero \ne \bsk \in A_p(\bsg)} (r_{1,\bsgamma}(\bsk))^{-1} 
	, \qquad
	T_{\alpha,\bsgamma}(\bsg,p)
	:= 
	\sum_{\bszero \ne \bsk \in A_p(\bsg)} (r_{\alpha,\bsgamma}(\bsk))^{-1}
\end{equation}      
with index set given by
\begin{equation*}
	A_p(\bsg)
	:=
	\{ \bsk \in \{0,1,\ldots,b^m-1\}^d \mid \bsk \in \calD (\bsg, p) \}
	.
\end{equation*}
Furthermore, for a subset $\emptyset\neq\setu \subseteq \{1 \mcol d\}$, we introduce the sets
\begin{align*}
	A_\setu 
	&=
	A_{p,\setu}(\bsg_\setu)
	=
	A_{p,\setu}(\bsg)
	:=
	\{ \bsk_\setu \in \{0,1,\ldots,b^m-1\}^{\abs{\setu}} \mid \bsk_\setu \in \calD_{\setu}(\bsg,p) \}
	, \\
	A^\ast_\setu 
	&=
	A^\ast_{p,\setu}(\bsg_\setu)
	=
	A^\ast_{p,\setu}(\bsg)
	:=
	\{ \bsk_\setu \in \{1,\ldots,b^m-1\}^{\abs{\setu}} \mid \bsk_\setu \in \calD_{\setu}(\bsg,p) \},
\end{align*}
and for a polynomial $p \in \F_b[x]$ define the indicator function $\delta_p: \F_b[x] \to \{0,1\}$ by
\begin{equation*}
	\delta_p(q)
	:=
	\begin{cases}
		1, & \text{if } q \equiv 0 \tpmod{p}, \\
		0, & \text{if } q \not\equiv 0 \tpmod{p} .
	\end{cases}
\end{equation*}

In the following proposition we estimate the difference between the worst-case error $e_{b^m,d,\alpha,\bsgamma}(\bsg)$
and the truncated quality measure $T_{\alpha,\bsgamma}(\bsg,p)$ of a polynomial lattice rule with generator $\bsg$ and modulus $p \in \F_b[x]$ with $\deg(p)=m$.

\begin{proposition} \label{prop:trunc_error}
	Let $\bsgamma = (\gamma_j)_{j\ge 1}$ be a sequence of positive weights, let $p \in \F_b[x]$ with $\deg(p)=m$, and let $\bsg = (g_1,\ldots,g_d) \in G_{b,m}^d$ such that $\gcd(g_j,p)=1$ for all $j=1,\ldots,d$. Then, for any $\alpha>1$ and $N=b^m$, we have
	\begin{align*}
		e_{b^m,d,\alpha,\bsgamma}(\bsg) - T_{\alpha,\bsgamma}(\bsg,p)
		\le
		\frac{1}{N^{\alpha}} \sum_{\emptyset\neq \setu \subseteq \{1:d\}} \gamma_{\setu} (2 \mu_b(\alpha))^{\abs{\setu}}
		.
	\end{align*}
\end{proposition}
\begin{proof}
	For a non-empty subset $\emptyset\neq\setu \subseteq \{1\mcol d\}$ and $i \in \{1\mcol d\}$, 
	we write for short $\bsk_{\setu\setminus\{i\}} \in \N_0^{\abs{\setu}-1}$ and 
	$\bsg_{\setu\setminus\{i\}} \in G_{b,m}^{\abs{\setu}-1}$ to denote the projections on the components in $\setu\setminus\{i\}$. The difference can then be rewritten as 
	\begin{align*}
		e_{b^m,d,\alpha,\bsgamma}(\bsg) - T_{\alpha,\bsgamma}(\bsg,p)
		=
		\sum_{\emptyset\neq \setu \subseteq \{1:d\}} \left(\sum_{\bsk_\setu \in \calD_{\setu}(\bsg_\setu)} (r_{\alpha,\bsgamma}(\bsk_\setu))^{-1} - \sum_{\bsk_\setu \in 
		A^\ast_{p,\setu}(\bsg_\setu)} 
		(r_{\alpha,\bsgamma}(\bsk_\setu))^{-1} \right), 
	\end{align*}
	motivating us to define the quantity
	\begin{equation*}
		S_{\alpha,\bsgamma,\setu}
		:= 
		\sum_{\bsk_\setu \in \calD_{\setu}(\bsg_\setu)} (r_{\alpha,\bsgamma}(\bsk_\setu))^{-1} - \sum_{\bsk_\setu \in A^\ast_{p,\setu}(\bsg_\setu)} (r_{\alpha,\bsgamma}(\bsk_\setu))^{-1} 
	\end{equation*}
	for $\emptyset \ne \setu\subseteq\{1:d\}$. In the following we distinguish two cases. \\ \\
	\noindent \textbf{Case 1:} Suppose that $\abs{\setu}=1$ such that $\setu = \{j\}$ for some $j \in \{1\mcol d\}$.
	Then, we have
	\begin{align*}
		S_{\alpha,\bsgamma,\{j\}}
		&=
		\sum_{\substack{k \in \N \\ \tr_m(k) \, g_j \equiv 0 \tpmod{p}}} (r_{\alpha,\gamma_j}(k))^{-1} 
		- \sum_{\substack{k \in \{1,\ldots,b^m-1\} \\ \tr_m(k) \, g_j \equiv 0 \tpmod{p}}} (r_{\alpha,\gamma_j}(k))^{-1}\\
		&=
		\sum_{\substack{k \ge b^m \\ \tr_m(k) \, g_j \equiv 0 \tpmod{p}}} (r_{\alpha,\gamma_j}(k))^{-1} .
	\end{align*}
	Note that $\tr_m(k) \, g_j \equiv 0 \tpmod{p}$ if and only if there is a $c \in \FF_b[x]$ such that $\tr_m(k) \, g_j = c p$ and thus, since
	$\gcd(g_j,p)=1$, we have that $\tr_m(k) = a p$ for some $a \in \F_b[x]$. But $\deg(\tr_m(k)) < m$ while $\deg(p)=m$, which implies that
	$\tr_m(k) = 0$ and thus $k = t \, b^m$ for some $t \in \N$. This yields
	\begin{align*}
		S_{\alpha,\bsgamma,\{j\}}
		&=
		\sum_{t=1}^{\infty} (r_{\alpha,\gamma_j}(t \, b^m))^{-1}
		=
		\gamma_j \sum_{t=1}^{\infty} b^{-\alpha \floor{\log_b{t \, b^m}}}
		=
		\gamma_j \sum_{t=1}^{\infty} b^{-\alpha \floor{m + \log_b{t}}} \\
		&=
		\gamma_j \sum_{t=1}^{\infty} b^{-\alpha m} \, b^{-\alpha \floor{\log_b{t}}}
		=
		\frac{\gamma_j}{b^{\alpha m}} \sum_{t=1}^{\infty} b^{-\alpha \psi_b(t)}
		=
		\gamma_j \frac{\mu_b(\alpha)}{b^{\alpha m}} .
	\end{align*}
	\noindent \textbf{Case 2:} Suppose that $\abs{\setu}\ge 2$. In this case, we find that
	\begin{equation*}
		S_{\alpha,\bsgamma,\setu} 
		\le
		\sum_{i\in\setu} \sum_{\bsk_{\setu\setminus \{i\}} \in \N^{\abs{\setu}-1}} \sum_{k_i \ge b^m} 
		\frac{\delta_p (\tr_m(k_i) g_i + \tr_m(\bsk_{\setu \setminus \{i\}})\cdot \bsg_{\setu \setminus \{i\}})}{r_{\alpha,\bsgamma}(\bsk_\setu)} 
		.
	\end{equation*}
	Then, for $\bsk_{\setu\setminus \{i\}} \in \N^{\abs{\setu}-1}$, we write $q = \tr_m(\bsk_{\setu \setminus \{i\}})\cdot \bsg_{\setu \setminus \{i\}}$,
	and estimate the expression
	\begin{align*}
		\sum_{k_i \ge b^m} \!\! \frac{\delta_p (\tr_m(k_i) g_i + q)}{r_{\alpha,\bsgamma}(\bsk_\setu)}
		&=
		\gamma_{\setu} \! \sum_{k_i \ge b^m} \frac{\delta_p (\tr_m(k_i) g_i + q)}{\prod_{j \in \setu} b^{\alpha \floor{\log_b{k_j}}}}\\
		&=
		\gamma_\setu \prod_{\substack{j\in\setu\\ j\ne i}} b^{-\alpha \floor{\log_b{k_j}}} \! \sum_{k_i \ge b^m} \frac{\delta_p (\tr_m(k_i) g_i + q)}{b^{\alpha \floor{\log_b{k_i}}}} \\
		&=
		\gamma_\setu \prod_{\substack{j\in\setu\\ j\ne i}} b^{-\alpha \floor{\log_b{k_j}}}
		\sum_{t=1}^\infty \sum_{k_i=t b^m}^{(t+1)b^m-1} \frac{\delta_p (\tr_m(k_i) g_i + q)}{b^{\alpha \floor{\log_b{k_i}}}} \\
		&\le
		\gamma_\setu \prod_{\substack{j\in\setu\\ j\ne i}} b^{-\alpha \floor{\log_b{k_j}}}
		\sum_{t=1}^\infty b^{-\alpha \floor{\log_b{t b^m}}} \underbrace{\sum_{k_i=t b^m}^{(t+1)b^m-1} \delta_p (\tr_m(k_i) g_i + q)}_{= 1} \\
		&=
		\gamma_\setu \prod_{\substack{j\in\setu\\ j\ne i}} b^{-\alpha \floor{\log_b{k_j}}}
		\sum_{t=1}^\infty b^{-\alpha \floor{m + \log_b t}}
		=
		\gamma_\setu \frac{\mu_b(\alpha)}{b^{\alpha m}} \prod_{\substack{j\in\setu\\ j\ne i}} b^{-\alpha \floor{\log_b{k_j}}} ,
	\end{align*}
	where the penultimate equality follows since if $\gcd(g_i,p)=1$ then for each $t$ and each $q \in \F_b[x]$ there exists exactly one 
	$k \in \{t b^m,\ldots, (t+1)b^m -1\}$ such that $\tr_m(k) g_i + q \equiv 0 \tpmod{p}$.
	
	Hence, we can estimate $S_{\alpha,\bsgamma,\setu}$, for $|\setu| \ge 2$, by
	\begin{align*}
			S_{\alpha,\bsgamma,\setu} 
			&\le
			\sum_{i\in\setu} \sum_{\bsk_{\setu\setminus \{i\}} \in \N^{\abs{\setu}-1}} 
			\gamma_\setu \frac{\mu_b(\alpha)}{b^{\alpha m}} \prod_{\substack{j\in\setu\\ j\ne i}} b^{-\alpha \floor{\log_b{k_j}}}
			=
			\gamma_\setu \frac{\mu_b(\alpha)}{b^{\alpha m}} \sum_{i\in \setu} \left(\sum_{k=1}^{\infty} b^{-\alpha \floor{\log_b{k}}} \right)^{\abs{\setu}-1} \\
			&=
			\gamma_\setu \frac{\mu_b(\alpha)}{b^{\alpha m}} \sum_{i\in \setu} \mu_b(\alpha)^{\abs{\setu}-1}
			=
			\gamma_\setu \frac{\mu_b(\alpha)^{\abs{\setu}}}{b^{\alpha m}} \abs{\setu}
			\le
			\gamma_\setu \frac{1}{N^\alpha} (2 \mu_b(\alpha))^{\abs{\setu}} 
			.
	\end{align*}
	In summary, we obtain, using the results for both cases from above,
	\begin{align*}
		\sum_{\emptyset\neq \setu \subseteq \{1:d\}} S_{\alpha,\bsgamma,\setu} 
		\le
		\frac{1}{N^{\alpha}} \sum_{\emptyset\neq \setu \subseteq \{1:d\}} \gamma_{\setu} (2 \mu_b(\alpha))^{\abs{\setu}}
		,
	\end{align*}
	which is the claimed upper estimate.
\end{proof}

Based on the previous result, it is straightforward to show the existence of good polynomial lattice rules with respect to the worst-case error in the weighted Walsh space,
if one assumes the modulus $p$ to be irreducible. We omit the proof, which uses standard methods.

\begin{theorem} \label{thm:exist_wce}
	Let $p \in \F_b[x]$ be an irreducible polynomial with $\deg(p)=m$, let $N=b^{m}$, and let $\bsgamma = (\gamma_j)_{j\ge 1}$ be positive weights.
	Then there exists a $\bsg \in G_{b,m}^d$ such that, for all $\alpha >1$, the worst-case error $e_{b^m,d,\alpha,\bsgamma}(\bsg)$ satisfies
	\begin{align*}
	    e_{b^m,d,\alpha,\bsgamma}(\bsg) 
	    \le 
	    \frac{1}{N^\alpha} \left(\sum_{\emptyset\neq \setu \subseteq \{1:d\}} \gamma_{\setu} (2\mu_b(\alpha))^{|\setu|} 
	    + \left(\sum_{\emptyset\neq \setu \subseteq \{1:d\}} \gamma_{\setu}^{1/\alpha}(m(b-1))^{|\setu|}\right)^{\alpha}\right)
	    .
	\end{align*}
\end{theorem}

Even though the result in Theorem \ref{thm:exist_wce} assures us that there always exist 
generating vectors of polynomial lattice point sets which are in a certain sense good, the result is not constructive. The 
road which we will take in the present paper is slightly different. Instead of assuming an irreducible modulus $p$, we will 
assume that $p$ has the special form $p(x)=p_m(x)=x^m$, and show a constructive approach to find generating vectors 
of good polynomial lattice rules. This will be the main result of our paper, which is stated in Theorem \ref{thm:main}.

\section{The CBC-DBD construction for polynomial lattice rules} \label{sec:cbc-dbd}

In this section, we formulate and analyze a method for the construction of good polynomial lattice rules. In contrast to the existence result in Theorem \ref{thm:exist_wce}, our construction method yields polynomial lattice rules with modulus $p(x)=x^m$. At first, we prove some auxiliary statements which will be needed in the further analysis.

\subsection{Preliminary results}

We consider the following Walsh series for $x\in (0,1)$, based on the decay function $r_1$,
\begin{equation*}
	\sum_{k=0}^\infty \frac{\wal_k(x)}{r_1(k)}
	=
	1 + \sum_{k=1}^\infty \frac{e^{2 \pi \icomp (\kappa_{0} \xi_1 + \kappa_{1} \xi_2 + \cdots )/b}}{b^{\floor{\log_b(k)}}}
	,
\end{equation*}
which, as we will see, is closely related to our quality criterion $T_{\bsgamma}$ introduced in \eqref{eq:qual_meas}. 
To this end, we define, for $n \in \N$, the $n$-th Walsh--Dirichlet kernel by
\begin{equation*} 
	D_n(x)=\sum_{k=0}^{n-1} { \wal_k(x)}.
\end{equation*}
From \cite[Lemma A.17]{DP10} it then follows that, for $x \in (0,1)$,
\begin{equation} \label{eq:w-d_kernel}
	D_{b^t}(x)
	=
	\left\{\begin{array}{cc}
		b^t, & \text{if } x\in (0, \frac{1}{b^t} ), \\
		0 , & \text{if } x\in [ \frac{1}{b^t}, 1 ).
	\end{array}\right.
\end{equation}
We can then prove the following identity.

\begin{lemma} \label{lem:Walsh-series}
	For base $b \ge 2$, the Walsh series of $-(b-1) (\floor{\log_b(x)} + 1)$ equals, pointwise for $x\in (0,1)$,
	\begin{equation*}
	-(b-1) (\floor{\log_b(x)} + 1)
	=
	1 + \sum_{k=1}^\infty \frac{e^{2 \pi \icomp (\kappa_{0} \xi_1 + \kappa_{1} \xi_2 + \cdots )/b}}{b^{\floor{\log_b(k)}}} 
	=
	\sum_{k=0}^\infty \frac{ \wal_k(x)}{r_{1}(k)} .
	\end{equation*}
\end{lemma} 
\begin{proof}
Using the definition of the Walsh-Dirichlet kernel, we obtain
\begin{equation*} 
	\sum_{k=1}^\infty \frac{\wal_{k}(x)}{b^{\lfloor\log_b(k)\rfloor}} 
	=
	\sum_{t=1}^{\infty} \sum_{k=b^{t-1}}^{b^t-1} \frac{ \wal_{k}(x)}{b^{t-1}}
	=
	\sum_{t=1}^{\infty}\frac{D_{b^t}(x)- D_{b^{t-1}}(x)}{b^{t-1}},
\end{equation*}
and from \eqref{eq:w-d_kernel} we find that for $t \ge 1$ we have
\begin{equation} \label{eq:diff_w-d_kernel}
	D_{b^t}(x)- D_{b^{t-1}}(x)
	=
	\left\{\begin{array}{cc}
		(b-1)b^{t-1}, & \text{if } x\in \left(0, \frac{1}{b^t} \right), \\
		-b^{t-1}, & \text{if } x\in \left[\frac{1}{b^t}, \frac{1}{b^{t-1}} \right), \\
		0, & \text{if } x\in \left[\frac{1}{b^{t-1}}, 1 \right).
	\end{array}\right.
\end{equation}
Applied inductively, the relation in \eqref{eq:diff_w-d_kernel} yields that for $x\in [\frac{1}{b^t}, \frac{1}{b^{t-1}})$ we have
\begin{align*} 
	\sum_{\ell=1}^{\infty}\frac{D_{b^\ell}(x)- D_{b^{\ell-1}}(x)}{b^{\ell-1}}
	&=
	\sum_{\ell=1}^{t-1} \frac{(b-1)b^{\ell-1}}{b^{\ell-1}} - \frac{b^{t-1}}{b^{t-1}}
	=
	(t-1)(b-1) - 1
\end{align*}
for all $t \ge 1$, which is equivalent to
\begin{equation*} 
	1 + \sum_{k=1}^\infty \frac{\wal_{k}(x)}{b^{\floor{\log_b(k)}}}
	=
	(b-1)(t-1)
	=
	-(b-1)(-t+1)
	=
	-(b-1) (\floor{\log_b(x)} + 1)
\end{equation*}
for $x\in [\frac{1}{b^t}, \frac{1}{b^{t-1}})$ and for all $t \in \N$. This proves the claimed identity.
\end{proof}

Based on the previous result in Lemma \ref{lem:Walsh-series}, we show that the function $-(b-1) (\floor{\log_b(x)} + 1)$ 
can be written in terms of its truncated Walsh series with uniformly bounded remainder term.

\begin{lemma}\label{lem:trunc_walsh_series}
	Let $N=b^m$ with $m\in \N$ and base $b \ge 2$. Then for any $x \in (0,1)$ there exists a $\tau = \tau(x)\in\R$ with 
	$|\tau(x)|< \frac{b}{b-1}$ such that
	\begin{equation} \label{eq:trunc_walsh_series}
		-(b-1) (\floor{\log_b(x)} + 1)
		=
		\sum_{k=0}^{N-1} \frac{\wal_k(x)}{r_{1}(k)} + \frac{\tau(x)}{N x}
		.
	\end{equation}
\end{lemma}
 
\begin{proof}
	The expansion in Lemma \ref{lem:Walsh-series} allows us to write
	\begin{equation*}
		-(b-1) (\floor{\log_b(x)} + 1)
		=
		\sum_{k=0}^{N-1} \frac{\wal_k(x)}{r_{1}(k)} + R_{N}(x)
		,
	\end{equation*}
	where the remainder $R_{N}(x)$ has the form
	\begin{align*}
		R_{N}(x)
		&= 	
		\sum_{k=b^m}^{\infty} \frac{\wal_k(x)}{r_{1}(k)} 
		=
		\sum_{k=b^m}^{\infty} \frac{\wal_k(x)}{b^{\floor{\log_b(k)}}}
		=
		\sum_{t=m}^{\infty}\frac{D_{b^{t+1}}(x)- D_{b^t}(x)}{b^t}
		.
	\end{align*}
	From \eqref{eq:diff_w-d_kernel} we then see that the following inequality holds,
	\begin{equation*} 
		\abs{D_{b^{t+1}}(x)- D_{b^t}(x)}
		< 
		\frac{1}{x}, \quad t\in\N, \ x \in (0,1),
	\end{equation*}
	and thus we obtain
	\begin{equation*}
		\abs{R_{N}(x)}
		=
		\abs{\sum_{t=m}^{\infty}\frac{D_{b^{t+1}}(x)- D_{b^t}(x)}{b^t}}
		<
		\frac{1}{x}\sum_{t=m}^{\infty}\frac{1}{b^t}
		=
		\frac{b}{(b-1)b^m}\frac{1}{x}
		=
		\frac{b}{(b-1)N x},
	\end{equation*}
	which implies the existence of a $\tau(x) \in \R$ with $|\tau(x)|< \frac{b}{b-1}$ such that the identity \eqref{eq:trunc_walsh_series} holds.
\end{proof}

\begin{remark}
	Using a more involved argument, the result in Lemma \ref{lem:trunc_walsh_series} can also be extended to general $N \in \N$. 
	In particular, we obtain that for any $x \in (0,1)$ there exists a $\tau=\tau(x) \in \R$ such that
	\begin{equation*}
		-(b-1) (\floor{\log_b(x)} + 1)
		=
		\sum_{k=0}^{N-1} \frac{\wal_k(x)}{r_{1}(k)} + \frac{\tau}{N x}
	\end{equation*}
	with $\abs{\tau} < b \left( \frac{1}{b-1} + 2 \right)$ for $b=2$ and with $\abs{\tau} < b \left( \frac{1}{b-1} + 2b \right)$ for $b>2$.
\end{remark}

We will also make use of the following lemma, which was proved in \cite{EKNO2020}.

\begin{lemma} \label{lem:diff_prod}
	For $j\in\{1\mcol d\}$, let $u_j, v_j$, and $\rho_j$ be real numbers satisfying
	\begin{align*}
		(a) \quad u_j = v_j + \rho_j, \quad
		(b) \quad |u_j| \le \bar{u}_j, \quad
		(c) \quad \bar{u}_j \geq 1 ,
	\end{align*}
	for all $j\in\{1 \mcol d\}$. Then, for any subset $\emptyset \ne \setu \subseteq \{1 \mcol d\}$ there exists a $\theta_{\setu}$ with $|\theta_{\setu}| \le 1$ such that 
	\begin{align*}
		\prod_{j \in \setu} u_j
		&=
		\prod_{j \in \setu} v_j + \theta_{\setu} \left(\prod_{j \in \setu} (\bar{u}_j+|\rho_j|) \right) \sum_{j \in \setu} |\rho_j|
		.
	\end{align*}
\end{lemma}

Furthermore, we recall the character property of Walsh functions for polynomial lattice rules with prime base $b$. Let $P(\bsg,p)=\{\bsx_0,\ldots,\bsx_{b^m-1}\}$ be a polynomial lattice with generating vector $\bsg \in (\F_b[x])^d$ and modulus $p \in \F_b[x]$ with $\deg(p)=m$. Then, for any integer vector $\bsk \in \N_0^d$ the following identity holds,
\begin{equation} \label{eq:char-prop}
	\frac{1}{b^m} \sum_{n=0}^{b^m-1} \wal_{\bsk} (\bsx_n)
	=
	\delta_p(\tr_m(\bsk) \cdot \bsg)
	=
	\begin{cases}
		1, & \text{if } \tr_m(\bsk) \cdot \bsg \equiv 0 \tpmod{p} , \\
		0, & \text{otherwise} .
	\end{cases}
\end{equation}
We remark that an analogous result to \eqref{eq:char-prop} also holds if we only consider projections of the polynomial lattice and the generating vector onto 
a non-empty subset of $\{1 \mcol d\}$, as also the projection of a polynomial lattice is a polynomial lattice that is generated by the corresponding projection of the generating vector.

We now state an auxiliary result that will be useful at several instances in this paper. 
\begin{lemma}\label{lem:projections_PLR}
	Let $P(\bsg,p)$ be a polynomial lattice with modulus $p \in \F_b[x]$ with $\deg(p)=m$ and generating vector 
	$\bsg=(g_1,\ldots,g_d) \in (\F_b[x])^d$ such that $\gcd(g_j,p)=1$ for $1 \le j \le d$. Then each one-dimensional projection of $P(\bsg,p)$ is the full grid 
	\begin{equation*}
		\left\{0,\frac{1}{b^m}, \ldots, \frac{b^m-1}{b^m}\right\} ,
	\end{equation*}
	and in particular the projection of the point with index $0$ is always $0$. 
\end{lemma}

\begin{proof}
	The result follows from Definition \ref{def:poly_lat} and \cite[Remark 10.3]{DP10}.
\end{proof}

Additionally, we will need the following result.

\begin{lemma}\label{lem:sum_PLR}
	Let $P(\bsg,p)=\{\bsx_0,\ldots,\bsx_{b^m-1}\}$ be a polynomial lattice point set with modulus $p \in \F_b[x]$ with $\deg(p)=m$ and generating vector $\bsg \in (\F_b[x])^d$ such that $\gcd(g_j,p)=1$ for $1 \le j \le d$. Furthermore, let $m\ge 4$. For a point $\bsx_n$ with $n\in\{0,1,\ldots,b^m-1\}$, we denote its coordinates via 
	$\bsx_n = (x_{n,1},\ldots,x_{n,d})$. Then, for any $j \in \{1,\ldots,d\}$, it is true that
	\begin{equation*}
		\frac{1}{b^m} \sum_{n=1}^{b^m-1} \frac{1}{x_{n,j}}
		<
		1 + m \ln b 
		\le 
		m (b-1)
		.
	\end{equation*}
\end{lemma}

\begin{proof}
	We recall that the point set $P(\bsg,p)$ is defined as the collection of the $b^m$ points of the form
	\begin{equation*}
		\bsx_n
		=
		\left(v_m \left( \frac{n(x)\, g_1(x)}{p(x)} \right),\ldots, v_m\left( \frac{n(x)\, g_d(x)}{p(x)} \right)\right)
	\end{equation*}
	for $n \in \F_b[x]$ with $\deg(n) < m$. Due to Lemma \ref{lem:projections_PLR} we know that
	$\{x_{1,j},\ldots,x_{b^m-1,j}\}$ equals the set $\left\{\frac{1}{b^m}, \ldots, \frac{b^m-1}{b^m}\right\}$ for each $j\in\{1,\ldots,d\}$. 
	Thus we can estimate
	\begin{align*}
		\frac{1}{b^m} \sum_{n=1}^{b^m-1} \frac{1}{x_{n,j}}
		&=
		\frac{1}{b^m} \sum_{n=1}^{b^m-1} \frac{b^m}{n}
		=
		\sum_{n=1}^{b^m-1} \frac1n
		\le
		1 + \int_{1}^{b^m-1} \frac{1}{x} \rd x \\
		&
		=
		1 + \ln(b^m-1)
		<
		1 + \ln(b^m)
		=
		1 + m \ln b
		\le
		m (b-1)
		,
	\end{align*}
	which yields the claimed result, where the last estimate follows from the assumption $m\ge 4$. 
\end{proof}

\subsection{The CBC-DBD construction algorithm}\label{sec:algorithm}

We are now ready to study the component-by-component digit-by-digit (CBC-DBD) construction for polynomial lattice rules, see also \cite{EKNO2020}, where such an algorithm was analyzed for ordinary lattice rules. In particular, we will assume throughout this section that our modulus polynomial is of the form $p_m (x)=x^m$ for $m\in\N$. 

Concerning the weights, the algorithm can, as indicated in our main result (Theorem \ref{thm:main}), be run with respect to the weights $\bsgamma^{1/\alpha} = (\gamma_j^{1/\alpha})_{j \ge 1}$ to obtain a polynomial lattice rule that yields a low worst-case error in the Walsh space $W_{d,\bsgamma}^\alpha$, or, alternatively, with respect to the weights $\bsgamma$ to obtain good polynomial lattice rules in the space $W_{d,\bsgamma^{\alpha}}^\alpha$. In the latter case, the construction algorithm is independent of the smoothness parameter $\alpha$ and we obtain worst-case error bounds that hold for all $\alpha>1$ simultaneously.

In order to avoid confusion, we will therefore denote the weights in this section by $\bseta$ instead of $\bsgamma$ and outline the algorithm based on $\bseta$. 
In Theorem \ref{thm:main}, we will then choose $\bseta$ equal to $\bsgamma^{1/\alpha}$ or $\bsgamma$, respectively.
For technical reasons, it will be necessary to assume that the positive weights $\bseta$ are of product structure, that is,
\begin{equation*}
	\eta_{\setu}
	=
	\prod_{j \in \setu} \eta_j
\end{equation*}
for $\setu\subseteq \{1 \mcol d\}$, with a sequence of positive reals $(\eta_j)_{j\ge1}$. However, we point out that the following theorem, which 
is crucial for the proposed construction method, also holds for general weights $\bseta = (\eta_{\setu})_{\setu \subseteq \{1:d\}}$.

\begin{theorem} \label{thm:CBCDBD_split}
	Let $b$ be prime, let $m,d \in \N$ with $m\ge 4$, let $p_m(x)=x^m \in \F_b[x]$, and let $\bseta = (\eta_\setu)_{\setu\subseteq\{1:d\}}$ be positive weights with 
	$\eta_{\emptyset} = 1$. Furthermore, let $\bsg=(g_1,\ldots, g_d) \in (\F_b[x])^d$ with $\deg(g_j) < m$ and $\gcd(g_j,p_m)=1$ for $1\le j\le d$. Then, 
	\begin{align*}
		T_{\bseta} (\bsg,p_m)
		&\le
		\frac{1}{b^m} \, H_{d,m,\bseta}(\bsg) - \sum_{\emptyset\ne \setu \subseteq \{1:d\}} \eta_\setu
		+ \sum_{\emptyset\ne \setu \subseteq \{1:d\}} \frac{\eta_\setu}{b^m} ((b-1)m+1)^{|\setu|}
		\notag \\ 
		&\quad+ 
		\sum_{\emptyset\ne \setu \subseteq \{1:d\}} \frac{\eta_\setu}{b^m} \, (b \,m \abs{\setu}) \left((b-1)m + \frac{b}{b-1} \right)^{\abs{\setu}}
		,
	\end{align*}
	where we define the function $H_{d,m,\bseta}: (\F_b[x])^d \to \R$ as
	\begin{equation} \label{eq:def_H}
		H_{d,m,\bseta}(\bsg)
		:=
		\sum_{\emptyset\ne \setu \subseteq \{1:d\}}\eta_\setu (1-b)^{|\setu|}
		\sum_{n=1}^{b^m-1} \prod_{j\in\setu} \left( \floor{\log_b \left(v_m \left( \frac{n(x) \, g_j(x)}{x^m} \right) \right)} + 1 \right).
	\end{equation}
\end{theorem} 

\begin{proof}
	We use the character property of Walsh functions in \eqref{eq:char-prop} to rewrite $T_{\bseta}(\bsg,p_m)$ with the help of the 
	identity in Lemma \ref{lem:trunc_walsh_series}. First, we recall that for $k \in \N_0$ we have
	\begin{equation*}
		r_1(k) =r_1(b,k)
		= 
		\left\{\begin{array}{cc}
		1, & {\text{for }} k = 0, \\
				b^{\floor{\log_b(k)}}, & {\text{for }} k \ne 0 .
		\end{array}\right.
	\end{equation*}
	Using this definition, we obtain that
	\begin{align} \label{eq:estimate_T_2}
		T_{\bseta} (\bsg,p_m)
		&=
		\!\! \sum_{\emptyset\ne \setu \subseteq \{1:d\}} \!\! \eta_\setu \sum_{\bsk_\setu \in \{1,\dots,b^m-1 \}^{|\setu|}}
		\frac{\delta_{p_m}(\tr_m(\bsk_\setu) \cdot \bsg_\setu)}{\prod_{j\in\setu} r_1(k_j)}
		\notag \\
		&\le
		\!\! \sum_{\emptyset\ne \setu \subseteq \{1:d\}} \!\! \eta_\setu \sum_{\bszero \ne \bsk_\setu \in \{0,\dots,b^m-1 \}^{|\setu|}}
		\frac{\delta_{p_m}(\tr_m(\bsk_\setu) \cdot \bsg_\setu)}{\prod_{j\in\setu} r_1(k_j)}
		\notag \\
		&=
		\!\! \sum_{\emptyset\ne \setu \subseteq \{1:d\}} \frac{\eta_\setu}{b^m} \sum_{n=0}^{b^m-1}
		\left[ \sum_{\bsk_\setu \in \{0,1,\dots,b^m-1\}^{|\setu|}} \frac{\wal_{\bsk_\setu}(\bsx_{n,\setu})}{\prod_{j\in\setu} r_1(k_j)} - 1 \right]
		\notag \\
		&=
		\!\! \sum_{\emptyset\ne \setu \subseteq \{1:d\}} \!\! \frac{\eta_\setu}{b^m}
		\left[ \sum_{\bsk_\setu \in \{0,\dots,b^m-1\}^{|\setu|}} \frac{1}{\prod_{j\in\setu} r_1(k_j)} 
		+ \sum_{n=1}^{b^m-1} \prod_{j \in \setu} \left(1 + \sum_{k=1}^{b^m-1} \frac{\wal_k(x_{n,j})}{b^{\floor{\log_b(k)}}} \right) \right]
		\! - \!\!\!\!\sum_{\emptyset\ne \setu \subseteq \{1:d\}} \!\!\!\!\!\eta_\setu
		\notag \\ 
		&=
		\!\! \sum_{\emptyset\ne \setu \subseteq \{1:d\}} \frac{\eta_\setu}{b^m}
		\left[ \sum_{\bsk_\setu \in \{0,\dots,b^m-1\}^{|\setu|}} \frac{1}{\prod_{j\in\setu} r_1(k_j)} 
		+ \sum_{n=1}^{b^m-1} \left[ \prod_{j\in\setu} v_j(n) - \prod_{j\in\setu} u_j(n) + \prod_{j\in\setu} u_j(n) \right] \right]
		\notag \\
		&\quad- \sum_{\emptyset\ne \setu \subseteq \{1:d\}} \eta_\setu
		\notag \\ 
		\begin{split}
			&= 
			\sum_{\emptyset\ne \setu \subseteq \{1:d\}} \frac{\eta_\setu}{b^m} 
			\sum_{\bsk_\setu \in \{0,1,\dots,b^m-1 \}^{|\setu|}} \frac{1}{\prod_{j\in\setu} r_1(k_j) } 
			+ \sum_{\emptyset\ne \setu \subseteq \{1:d\}} \frac{\eta_\setu}{b^m} \sum_{n=1}^{b^m-1} \prod_{j\in\setu} u_j(n)
			\\
			&\quad-
			\sum_{\emptyset\ne \setu \subseteq \{1:d\}} \frac{\eta_\setu}{b^m} 
			\sum_{n=1}^{b^m-1} \theta_{\setu}(n) \left( \prod_{j\in\setu} (\bar{u}_j + |\rho_j(n)|) \right) \sum_{j\in\setu} |\rho_j(n)|
			- \sum_{\emptyset\ne \setu \subseteq \{1:d\}} \eta_\setu ,
		\end{split}
	\end{align}
	where we used Lemma \ref{lem:diff_prod} with
	\begin{align*}
		u_j = u_j(n) &:= -(b-1) (\floor{\log_b(x_{n,j})} + 1) ,
		&
		\bar{u}_j=\bar{u}_j(n) &:= (b-1)m ,
		\\
		v_j = v_j(n) &:= 1 + \sum_{k=1}^{b^m-1} \frac{\wal_k(x_{n,j})}{b^{\floor{\log_b(k)}}} ,
		&
		\rho_j = \rho_j(n) &:= \frac{\tau_j(n)}{x_{n,j} \, b^m}
		,
	\end{align*}
	and all $|\theta_\setu(n)| \le 1$ and $|\tau_j(n)| < \frac{b}{b-1}$. Due to Lemma \ref{lem:trunc_walsh_series}, Condition~(a) of Lemma \ref{lem:diff_prod} is fulfilled.
	Furthermore, we see that for $p_m(x)=x^m$ we have for each $j \in \{1\mcol d\}$ that
	\begin{equation*}
		x_{n,j} 
		=
		v_m\left(\frac{n(x) g_j(x)}{p_{m}(x)}\right)
		\ge
		v_m\left(\frac{1}{x^m}\right)
		=
		b^{-m}
	\end{equation*}
	for every $1 \le n < b^m$, and so
	\begin{equation*}
		\abs{u_j(n)}
		\le 
		-(b-1) (\floor{\log_b(b^{-m})} + 1)
		=
		-(b-1) (-m + 1)
		<
		(b-1) m 
		=
		\bar{u}_j 
	\end{equation*}
	with $\bar{u}_j \ge 1$ such that also Conditions~(b) and~(c) of Lemma \ref{lem:diff_prod} are satisfied.

	By simple calculations, the first sum in \eqref{eq:estimate_T_2} can be shown to equal
	\begin{align*}
		& \sum_{\emptyset\ne \setu \subseteq \{1:d\}} \frac{\eta_\setu}{b^m} 
		\sum_{\bsk_\setu \in \{0,1,\dots,b^m-1 \}^{|\setu|}} \frac{1}{\prod_{j\in\setu} r_1 (k_j)}
		= 
		\frac{1}{b^m}  \sum_{\emptyset\ne \setu \subseteq \{1:d\}} \eta_\setu 
		\left(1+\sum\limits_{k=1}^{b^{m}-1}\frac{1}{ b^{\floor{\log_b(k)}}} \right)^{|\setu|}	
		 \\
		&=
		 \frac{1}{b^m}  \sum_{\emptyset\ne \setu \subseteq \{1:d\}} \eta_\setu 
		 \left(	1+\sum\limits_{t=0}^{m-1}\sum\limits_{k=b^t}^{b^{t+1}-1}\frac{1}{ b^{\floor{\log_b(k)}}} \right)^{|\setu|}
		=
		\frac{1}{b^m} \sum_{\emptyset\ne \setu \subseteq \{1:d\}} \eta_\setu ((b-1)m+1)^{|\setu|}
		,
	\end{align*}
	while the third sum in \eqref{eq:estimate_T_2} can be bounded by 
	\begin{align*}
		&-\sum_{\emptyset\ne \setu \subseteq \{1:d\}} \frac{\eta_\setu}{b^m} 
		\sum_{n=1}^{b^m-1} \theta_{\setu}(n) 
		\left( \prod_{j\in\setu} (\bar{u}_j + |\rho_j(n)|) \right) \sum_{j\in\setu} |\rho_j(n)|
 		\\
		&\quad=
		-\sum_{\emptyset\ne \setu \subseteq \{1:d\}} \frac{\eta_\setu}{b^m}
		\sum_{n=1}^{b^m-1} \theta_{\setu}(n) \left( \prod_{j\in\setu} \left((b-1)m + \frac{|\tau_j(n)|}{x_{n,j} \, b^m}\right) \right) 
		\sum_{j\in\setu} \frac{|\tau_j(n)|}{x_{n,j} \, b^m}
		\notag \\
		&\quad\le
		\sum_{\emptyset\ne \setu \subseteq \{1:d\}} \frac{\eta_\setu}{b^m}
		\sum_{n=1}^{b^m-1} |\theta_{\setu}(n) | \left( \prod_{j\in\setu} \left((b-1)m + \frac{b}{b-1} \right) \right) 
		\sum_{j\in\setu} \frac{|\tau_j(n)|}{x_{n,j} \, b^m}
		\notag \\
		&\quad\le
		\sum_{\emptyset\ne \setu \subseteq \{1:d\}} \frac{\eta_\setu}{b^m}
		\left( \prod_{j\in\setu} \left((b-1)m + \frac{b}{b-1} \right) \right) 
		\sum_{j\in\setu} \sum_{n=1}^{b^m-1} \frac{b}{(b-1) b^m} \frac{1}{x_{n,j}}
		\notag \\
		&\quad\le	
		\sum_{\emptyset\ne \setu \subseteq \{1:d\}} \frac{\eta_\setu}{b^m}
		\left( \prod_{j\in\setu} \left((b-1)m + \frac{b}{b-1} \right) \right) \frac{b}{b-1}
		\sum_{j\in\setu} m (b-1)
		\notag \\
		&\quad=
		\frac{1}{b^m} \sum_{\emptyset\ne \setu \subseteq \{1:d\}} \eta_\setu \, (b \,m \abs{\setu})
		\left((b-1)m + \frac{b}{b-1} \right)^{\abs{\setu}}
		,
	\end{align*}
	where we used Lemma \ref{lem:sum_PLR} and the fact that $x_{n,j} \ge b^{-m}$ for each $j$ and all $1 \le n < b^m$.
	Combining these results with \eqref{eq:estimate_T_2} yields the claimed result.
\end{proof}

Theorem \ref{thm:CBCDBD_split} implies that it essentially suffices to find a generating vector $\bsg \in (\F_b[x])^d$ such that $H_{d,m,\bseta}(\bsg)$ is small, 
which then implies that also a good bound on $T_{\bseta}(\bsg,p_m)$ holds. We will therefore consider the quantity $H_{d,m,\bseta}$ as a search criterion 
for good generating vectors. 

At first, we prove the following result which will be needed in the further analysis and remind the reader that by $p_m$ we denote the polynomial $p_m\in\F_b[x]$ with $p_m(x)=x^m$ for $m\in\N$. 

\begin{lemma} \label{lem:bit-average}
	Let a prime $b$, an integer $t \ge 2$, and polynomials $\ell, q \in \F_b[x]$ with $\gcd(\ell,p_1)=\gcd(q,p_1)=1$ be given. Then the following identity holds:
	\begin{equation*}
		\sum_{g \in \F_b} \left( \floor{\log_b\left( v_t\left( \frac{\ell (x)\,(q(x) + x^{t-1} g)}{x^t} \right) \right)} + 1 \right)
		=
		\floor{\log_b\left( v_{t-1}\left( \frac{\ell(x) \, q(x)}{x^{t-1}} \right) \right)}
		.
	\end{equation*} 
\end{lemma}  

\begin{proof}
	Assume that the product of the polynomials $\ell$ and $q$ is given by
	\begin{equation*}
		\ell (x) \, q (x)
		=
		\sum_{i=0}^r a_i x^i 
		\quad \text{with} \quad
		a_0,a_r \ne 0
		.
	\end{equation*}
	Let, furthermore,
	\begin{equation*}
		\ell (x)=\sum_{k=0}^v \ell_k x^k
		,
	\end{equation*}
	where we note that $v\le r$. 
	Hence, we obtain that for $g \in \F_b$
	\begin{equation*}
		\frac{\ell(x) \,(q (x) + x^{t-1} g)}{x^t}
		=
		\sum_{i=t}^r a_i x^{i-t} + \sum_{k=1}^v \ell_k g x^{k-1} +  (a_{t-1} + \ell_0 g) x^{-1} + \sum_{i=0}^{t-2} a_i x^{i-t}
	\end{equation*}
	and thus we have that if $a_{t-1} + \ell_0 g \not\equiv 0 \pmod{b}$, then
	\begin{equation*}
		\floor{\log_b\left( v_t\left( \frac{\ell(x) \,(q (x) + x^{t-1} g)}{x^t} \right) \right)} + 1
		=
		\floor{\log_b\left( \frac{a_{t-1} + \ell_0 g}{b} + \sum_{i=0}^{t-2} a_i b^{i-t} \right)} + 1
		=
		- 1 + 1
		=
		0
		. 
	\end{equation*}
	Otherwise, if $a_{t-1} + \ell_0 g \equiv 0 \pmod{b}$, then
	\begin{align*}
		v_t \left( \frac{\ell (x) \,(q (x) + x^{t-1} g)}{x^t} \right)
		&=
		\sum_{i=0}^{t-2} a_i b^{i-t}
		=
		\sum_{i=2}^t a_{t-i} b^{-i}
		=
		\frac1b \left( \sum_{i=1}^{t-1} a_{t-i-1} b^{-i} \right)
		=
		\frac1b v_{t-1}\left( \frac{\ell (x) \, q (x)}{x^{t-1}} \right)
		,
	\end{align*}
	and therefore
	\begin{align*}
		\floor{\log_b\left( v_t \left( \frac{\ell (x) \,(q (x) + x^{t-1} g)}{x^t} \right) \right)} + 1
		&=
		\floor{\log_b\left( \frac1b v_{t-1}\left( \frac{\ell (x) \, q (x)}{x^{t-1}} \right) \right)} + 1\\
		&=
		\floor{\log_b\left( v_{t-1}\left( \frac{\ell (x) \, q (x)}{x^{t-1}} \right) \right)}
		.
	\end{align*}
	Observing that there exists exactly one $g \in \F_b$ for which $a_{t-1} + \ell_0 g \equiv 0 \pmod{b}$ and combining the two 
	cases considered, we immediately obtain the claimed identity.
\end{proof}

With the help of Lemma \ref{lem:bit-average} we can prove the following result which motivates the choice of our quality function for Algorithm \ref{alg:cbcdbd}.

\begin{lemma} \label{lem:H-average}
	For integers $m \in \N$ and $w \in \{1,\ldots,m\}$, let $b$ be prime, $g \in \F_b$, and $\bsg \in (\F_b[x])^d$ with $\gcd(g_j,p_1)=1$ for all $1 \le j \le d$, where 
	$g_d \in G_{b,w-1}$, and let $\bseta = (\eta_\setu)_{\setu\subseteq\{1:d\}}$ be positive weights with $\eta_{\emptyset} = 1$.
	Then the average of $H_{d,m,\bseta}$ with respect to the choices for extending the degree of $g_d + g \,p_{w-1}$ up to $m$ equals
	\begin{align} \label{eq:H-average}
		&\frac{1}{b^{m-w}} \sum_{\bar{g} \in G_{b,m-w}} H_{d,m,\bseta}(g_1,\ldots,g_{d-1},g_d + g \,p_{w-1} + \bar{g} \,p_w)
		\notag \\
		\begin{split}
		&\quad=
		\sum_{t=w}^{m} \sum_{\substack{\ell=1 \\ \ell \not\equiv 0 \tpmod b }}^{b^t-1}
		\eta_d (1-b) \, b^{w-t} \left( \floor{\log_b \left(v_w \left( \frac{\ell (x) \, (g_d (x) + g \,x^{w-1})}{x^w} \right) \right)} + 1 \right) \times
		\\
		&\qquad\times \prod_{j=1}^{d-1} \left(1 + \eta_j (1-b) \left( \floor{\log_b \left(v_t \left( \frac{\ell(x) \, g_j(x)}{x^t} \right) \right)} + 1 \right) \right)
		+ S_{m,w,\bseta}(\bsg) - (b^m-1)
		,
		\end{split}
	\end{align}
	where the term $S_{m,w,\bseta}(\bsg)$, which does not depend on $g$ and $\bar{g}$, is given by
	\begin{align*}
		S_{m,w,\bseta}(\bsg)
		&=
		\sum_{t=1}^{w-1} \sum_{\substack{\ell=1 \\ \ell \not\equiv 0 \tpmod b }}^{b^t-1} \prod_{j=1}^d \left(1 + \eta_j (1-b) 
		\left( \floor{\log_b \left(v_t \left( \frac{\ell(x) \, g_j(x)}{x^t} \right) \right)} + 1 \right) \right)
		\\
		&\quad+
		\sum_{t=w}^m \sum_{\substack{\ell=1 \\ \ell \not\equiv 0 \tpmod b }}^{b^t-1} 
		\prod_{j=1}^{d-1} \left(1 + \eta_j (1-b) \left(\floor{\log_b \left(v_t \left( \frac{\ell(x) \, g_j(x)}{x^t} \right) \right)} + 1 \right) \right) 
		\left( 1 + \eta_d (1-b^{w-t}) \right).
	\end{align*}
\end{lemma}

\begin{proof}
	For product weights $\eta_\setu = \prod_{j \in \setu} \eta_j$ and $\widetilde{\bsg}=(\widetilde{g}_1,\ldots,\widetilde{g}_d)\in (\F_b [x])^d$, 
	the quantity $H_{d,m,\bseta}(\widetilde{\bsg})$ defined in \eqref{eq:def_H} equals
	\begin{equation*}
		H_{d,m,\bseta}(\widetilde{\bsg})
		=
		\sum_{n=1}^{b^m-1} \prod_{j=1}^d \left(1 + \eta_j (1-b) \left( \floor{\log_b \left(v_m \left( \frac{n(x) \, 
		\widetilde{g}_j(x)}{x^m} \right) \right)} + 1 \right) \right) - (b^m-1).
	\end{equation*}
	We define $\bar{H}_{d,m,\bseta}(\widetilde{\bsg}) := H_{d,m,\bseta}(\widetilde{\bsg}) + (b^{m}-1)$ which in turn can be rewritten as
	\begin{equation*}
		\bar{H}_{d,m,\bseta}(\widetilde{\bsg})
		=
		\sum_{t=1}^{m} \sum_{\substack{\ell=1 \\ \ell \not\equiv 0 \tpmod b }}^{b^t-1} \prod_{j=1}^d \left(1 + \eta_j (1-b) 
		\left( \floor{\log_b \left(v_t \left( \frac{\ell(x) \, \widetilde{g}_j(x)}{x^t} \right) \right)} + 1 \right) \right).
	\end{equation*}
	Setting $\widetilde{g}_d=g_d + g \,p_{w-1} + \bar{g} \,p_w$ with $\bar{g}\in G_{b,m-w}$ and $\widetilde{g}_j=g_j$ for $j \in \{1 \mcol d-1\}$ , we can write
	\begin{eqnarray*}
		\lefteqn{\frac{1}{b^{m-w}} \sum_{\bar{g} \in G_{b,m-w}} H_{d,m,\bseta}(g_1,\ldots,g_{d-1},g_d + g \,p_{w-1} + \bar{g} \,p_w)
		=
		\frac{1}{b^{m-w}} \sum_{\bar{g} \in G_{b,m-w}} \bar{H}_{d,m,\bseta}(\widetilde{\bsg})- (b^m-1)}
		\\
		&=& \!\!\!
		\frac{1}{b^{m-w}} \!\!\! \sum_{\bar{g} \in G_{b,m-w}}
		\sum_{t=1}^{m} \sum_{\substack{\ell=1 \\ \ell \not\equiv 0 \tpmod b }}^{b^t-1} \prod_{j=1}^d \left(1 + \eta_j (1-b) 
		\left( \floor{\log_b \left(v_t \left( \frac{\ell(x) \, \widetilde{g}_j(x)}{x^t} \right) \right)} + 1 \right) \right)
		- (b^m-1)
		\\
		&=& \!\!\!
		\frac{1}{b^{m-w}} \!\!\! \sum_{\bar{g} \in G_{b,m-w}}
		\sum_{t=1}^{w-1} \sum_{\substack{\ell=1 \\ \ell \not\equiv 0 \tpmod b }}^{b^t-1} \prod_{j=1}^d \left(1 + \eta_j (1-b) 
		\left( \floor{\log_b \left(v_t \left( \frac{\ell(x) \, \widetilde{g}_j(x)}{x^t} \right) \right)} + 1 \right) \right)
		- (b^m-1)
		\\
		&&\!\!\!+\frac{1}{b^{m-w}} \sum_{\bar{g} \in G_{b,m-w}}
		\sum_{t=w}^{m} \sum_{\substack{\ell=1 \\ \ell \not\equiv 0 \tpmod b }}^{b^t-1} \prod_{j=1}^d \left(1 + \eta_j (1-b) 
		\left( \floor{\log_b \left(v_t \left( \frac{\ell(x) \,\widetilde{g}_j(x)}{x^t} \right) \right)} + 1 \right) \right)
		.
	\end{eqnarray*}
	The term $- (b^m-1)$ in \eqref{eq:H-average} is therefore accounted for. What is more, by the definition of $v_t$ we have for any $q \in \F_b[x]$ that
	\begin{equation} \label{eq:mod_prop}
		v_t \left( \frac{q(x)}{x^t} \right)
		=
		v_t \left( \frac{q(x) \bmod x^t}{x^t} \right)
		,
	\end{equation}
	and hence 
	\begin{eqnarray*}
		\lefteqn{\frac{1}{b^{m-w}} \sum_{\bar{g} \in G_{b,m-w}}
		\sum_{t=1}^{w-1} \sum_{\substack{\ell=1 \\ \ell \not\equiv 0 \tpmod b }}^{b^t-1} \prod_{j=1}^d \left(1 + \eta_j (1-b) 
		\left( \floor{\log_b \left(v_t \left( \frac{\ell(x) \, \widetilde{g}_j(x)}{x^t} \right) \right)} + 1 \right) \right)}
		\\
		&=&
		\sum_{t=1}^{w-1} \sum_{\substack{\ell=1 \\ \ell \not\equiv 0 \tpmod b }}^{b^t-1} \prod_{j=1}^d \left(1 + \eta_j (1-b) 
		\left( \floor{\log_b \left(v_t \left( \frac{\ell(x) \,g_j}{x^t} \right) \right)} + 1 \right) \right)
		,
	\end{eqnarray*}
	which is the first sum in $S_{m,w,\bseta}$, and, in particular, is independent of $g$ and all $\bar{g} \in G_{b,m-w}$. 
	
	The second sum in $S_{m,w,\bseta}$ and all remaining terms in identity \eqref{eq:H-average} are obtained by considering
	\begin{align} \label{eq:estimate_H_bar}
		&\frac{1}{b^{m-w}} \sum_{\bar{g} \in G_{b,m-w}} \sum_{t=w}^{m} \sum_{\substack{\ell=1 \\ \ell \not\equiv 0 \tpmod b }}^{b^t-1} 
		\prod_{j=1}^d \left(1 + \eta_j (1-b) \left( \floor{\log_b \left(v_t \left( \frac{\ell(x) \, \widetilde{g}_j(x)}{x^t} \right) \right)} + 1 \right) \right)
		\notag \\
		\begin{split}
		&=
		\sum_{t=w}^{m} \sum_{\substack{\ell=1 \\ \ell \not\equiv 0 \tpmod b }}^{b^t-1} 
		\prod_{j=1}^{d-1} \left(1 + \eta_j (1-b) \left( \floor{\log_b \left(v_t \left( \frac{\ell(x) \, g_j(x)}{x^t} \right) \right)} + 1 \right) \right) \times
		\\
		&\quad\times
		\left(1 + \eta_d (1-b) \frac{1}{b^{m-w}} \sum_{\bar{g} \in G_{b,m-w}} 
		\left(\floor{\log_b \left(v_t \left( \frac{\ell (x) \, (g_d (x) + g \,x^{w-1} + \bar{g}(x) \,x^w)}{x^t} \right) \right)}
		+ 1 \right) \right)
		,
		\end{split}
	\end{align}
	such that, with the help of \eqref{eq:mod_prop} and under the repeated use of Lemma $\ref{lem:bit-average}$, we obtain for each $t \in \{w+1,\ldots,m\}$ that
	\begin{align*}
		&\sum_{\bar{g} \in G_{b,m-w}} \left( \floor{\log_b \left(v_t \left( \frac{\ell (x) \, (g_d (x) + g \,x^{w-1} + \bar{g} (x) \,x^w)}{x^t} \right) \right)} + 1 \right)
		\\
		&\quad=
		b^{m-t} \sum_{\bar{g} \in G_{b,t-w}} \left( \floor{\log_b \left(v_t \left( \frac{\ell (x)\, (g_d (x) + g \,x^{w-1} + \bar{g} (x) \,x^w)}{x^t} \right) \right)} + 1 \right)
		\\
		&\quad=
		b^{m-t} \sum_{\bar{g} \in G_{b,t-w-1}} \left( \floor{\log_b \left(v_{t-1} \left( \frac{\ell (x)\, (g_d (x) + g \,x^{w-1} + \bar{g} (x) \,x^w)}{x^{t-1}} \right) \right)} + 1 - 1 \right)
		\\
		&\quad=
		b^{m-t} \left( \floor{\log_b \left(v_w \left( \frac{\ell (x) \, (g_d (x) + g \,x^{w-1})}{x^w} \right) \right)} + 1 \right) - b^m \sum_{r=w+1}^t b^{-r}
		\\
		&\quad=
		b^{m-w} b^{w-t} \left(\floor{\log_b \left(v_w \left( \frac{\ell (x) \, (g_d (x) + g \,x^{w-1})}{x^w} \right) \right)} + 1 \right) - b^{m-w} \left( \frac{1-b^{w-t}}{b-1} \right)
		.
	\end{align*}
	Combining this with the identity in \eqref{eq:estimate_H_bar} yields the remaining term of $S_{m,w,\bseta}$ and the first term in \eqref{eq:H-average} such that
	the claimed result is proved.
\end{proof}

We note that only the first term of \eqref{eq:H-average} in Lemma \ref{lem:H-average} depends on the $(w-1)$-th order term $g x^{w-1}$ of $g_d$.
Therefore, we can introduce the quality function for our algorithm which is based on the first term of \eqref{eq:H-average}, yet slightly 
adjusted by an additional summand that is independent of $g$ and $\bar{g}$.

\begin{definition}\label{def:h_r,w} (Digit-wise quality function)
	Let $q \in \F_b[x]$, with prime $b$, let $m,d \in \N$, and let $\bseta = (\eta_{\setu})_{\setu \subseteq \{1:d\}}$, where
	$\eta_\setu=\prod_{j \in \setu} \eta_j$ with positive reals $(\eta_j)_{j\ge1}$, be product weights.
	For integers $w \in \{1 \mcol m\}$, $r \in \{1 \mcol d\}$, and polynomials $g_1,\ldots,g_{r-1} \in \F_b[x]$ with $\gcd(g_j,p_1)=1$ for $j=1,\ldots,r-1$, we define the 
	quality function $h_{r,w,m,\bseta}: \F_b[x] \to \R$ as
	\begin{align*}
		h_{r,w,m,\bseta}(q)
		&:=
		\sum_{t=w}^m \frac{1}{b^{t-w}} \sum_{\substack{\ell=1 \\ \ell \not\equiv 0 \tpmod b }}^{b^t-1}
		\left( 1 + \eta_{r} (1-b) \left( \floor{\log_b \left(v_w \left( \frac{\ell(x) \, q(x) }{x^w} \right) \right)} + 1 \right) \right) \times
		\notag \\
		&\quad\times \prod_{j=1}^{r-1} \left(1 + \eta_j (1-b) \left( \floor{\log_b \left(v_t \left( \frac{\ell(x) \, g_j(x)}{x^t} \right) \right)} + 1 \right) \right)
		.
	\end{align*}
\end{definition}

\noindent
We remark that the function $h_{r,w,m,\bseta}$ directly depends on the polynomials $g_1,\ldots,g_{r-1}$ even though this is not visible in the notation.
In the remainder of this section, however, these polynomials will always be the components of the generating vector which were selected in the previous steps of our
algorithm. Based on the quality function $h_{r,w,m,\bseta}$, we formulate the component-by-component digit-by-digit algorithm.

\begin{algorithm}[H] 
	\caption{Component-by-component digit-by-digit algorithm}	
	\label{alg:cbcdbd}
	\vspace{5pt}
	\textbf{Input:} Prime number $b \ge 2$, integers $m,d \in \N$, and positive product weights $\bseta=(\eta_j)_{j \ge 1}$. \\
	\vspace{-10pt}
	\begin{algorithmic}
		\STATE Set $g_{1,m} = 1$ and $g_{2,1} = \cdots = g_{d,1} = 1$.
		\vspace{5pt}
		\FOR{$r=2$ \TO $d$}
			\FOR{$w=2$ \TO $m$}
				\STATE $g^{\ast} = \underset{g \in \F_b}{\argmin} \;  h_{r,w,m,\bseta}(g_{r,w-1} + g \,p_{w-1})$
				\STATE $g_{r,w} = g_{r,w-1} + g^{\ast} p_{w-1}$
			\ENDFOR
		\ENDFOR
		\vspace{5pt}
		\STATE Set $\bsg = (g_1,\ldots,g_d)$ with $g_r := g_{r,m}$ for $r=1,\ldots,d$.
	\end{algorithmic}
	\vspace{5pt}
	\textbf{Return:} Generating vector $\bsg=(g_1,\ldots,g_d) \in (G_{b,m}^\ast)^d$.
\end{algorithm}

In the next section, we study the worst-case error behavior of polynomial lattice rules with generating vectors obtained by Algorithm \ref{alg:cbcdbd}. 

\subsection{Error bounds for the constructed polynomial lattice rules}

The following theorem shows that for the constructed polynomial lattice rules the quantity $H_{d,m,\bseta}(\bsg)$, which for product weights 
$\eta_{\setu} = \prod_{j\in\setu} \eta_j$ equals
\begin{equation*}
	\begin{aligned}
		H_{d,m,\bseta}(\bsg)
		&=
		\sum_{n=1}^{b^m-1} \prod_{j=1}^d \left(1 + \eta_j (1-b) \left( \floor{\log_b \left(v_m \left( \frac{n(x) \, g_j(x)}{x^m} \right) \right)} + 1 \right) \right) - (b^m-1)
		\\
		&=
		\sum_{t=1}^{m} \sum_{\substack{\ell=1 \\ \ell \not\equiv 0 \tpmod b }}^{b^t-1} \prod_{j=1}^d \left(1 + \eta_j (1-b) \left( \floor{\log_b \left(v_t \left( \frac{\ell(x) \, g_j(x)}{x^t} \right) \right)} + 1 \right) \right) - (b^m-1)
		,
	\end{aligned}
\end{equation*}
can be related to the quantity $H_{d-1,m,\bseta}(\bsg_{\{1 \mcol d-1\}})$. 
  
\begin{theorem} \label{thm:H_s_s-1}
	Let $b$ be prime, $m,d \in \N$ be integers with $d\ge 2$, and let $\bseta = (\eta_j)_{j\ge 1}$ be positive product weights.
	Furthermore, denote by $\bsg$ the corresponding generating vector constructed by Algorithm \ref{alg:cbcdbd}. Then $\bsg$ satisfies
	\begin{equation} \label{eq:H_s_s-1}
		H_{d,m,\bseta}(\bsg)
		\le
		\left( 1 + \eta_d \right) H_{d-1,m,\bseta}(\bsg_{\{1:d-1\}}) + \eta_d (b^m-1)
		.
	\end{equation}
\end{theorem}
\begin{proof}
	We will prove \eqref{eq:H_s_s-1} by an inductive argument over the selection of the terms of order $1 \le t \le m-1$ of the polynomial $g_d \in \F_b[x]$.
	We start by considering the term of order $m-1$. According to Algorithm \ref{alg:cbcdbd}, this term has been selected by minimizing 
	$h_{d,m,m,\bseta}(g_{d,m-1} + g \,p_{m-1})$ over the choices $g \in \F_b$, and where $g_{d,m-1} \in G_{b,m-1}$ has been determined in the previous 
	steps of the algorithm. By Lemma \ref{lem:H-average} (with $w=m$) and Definition \ref{def:h_r,w} this is equivalent to minimizing
	\begin{equation*}
		H_{d,m,\bseta}(g_1,\ldots,g_{d-1},g_{d,m-1} + g \,p_{m-1})
	\end{equation*}
	with respect to $g \in \F_b$. By the standard averaging argument, this yields
	\begin{align} \label{eq:est_H_ind}
		H_{d,m,\bseta}(\bsg)
		&=
		\min_{\bar{g} \in \F_b} H_{d,m,\bseta}\left(\bsg_{\{1 \mcol d-1\}},g_{d,m-1} + \bar{g} \,p_{m-1} \right)
		\notag \\
		&\le
		\frac1b \sum_{\bar{g} \in \F_b} H_{d,m,\bseta}\left(\bsg_{\{1 \mcol d-1\}},g_{d,m-1} + \bar{g} \,p_{m-1} \right)
		\notag \\
		&=
		\frac1b \sum_{\bar{g} \in G_{b,1}} H_{d,m,\bseta}\left(\bsg_{\{1 \mcol d-1\}},g_{d,m-2} + g \,p_{m-2} + \bar{g} \,p_{m-1} \right)
		,
	\end{align}
	where $g_{d,m-1}$ has been split up into $g_{d,m-2}$ and $g \,p_{m-2}$ in accordance with Algorithm \ref{alg:cbcdbd} such that $g$ has been selected in the previous step of the algorithm and we used that $G_{b,1} \cong \F_b$. 
	
	Similarly, we observe that the term of order $m-2$ has been selected by minimizing $h_{d,m-1,m,\bseta}(g_{d,m-2} + g \,p_{m-2})$ with respect to 
	the choices $g \in \F_b$. Again, by Lemma \ref{lem:H-average} (with $w=m-1$) and Definition \ref{def:h_r,w} this is equivalent to minimizing
	\begin{equation*}
		\frac{1}{b} \sum_{\bar{g} \in G_{b,1}} H_{d,m,\bseta}(\bsg_{\{1 \mcol d-1\}},g_{d,m-2} + g \,p_{m-2} + \bar{g} \,p_{m-1})
	\end{equation*}
	with respect to $g \in G_{b,1} \cong \F_b$. By the standard averaging argument, we obtain that
	\begin{align*}
		&\min_{g \in G_{b,1}} \frac1b \sum_{\bar{g} \in \F_b} H_{d,m,\bseta}\left(\bsg_{\{1 \mcol d-1\}},g_{d,m-2} + g \,p_{m-2} + \bar{g} \,p_{m-1} \right)
		\\
		&\quad\le
		\frac{1}{b^2} \sum_{g \in \F_b} \sum_{\bar{g} \in G_{b,1}} H_{d,m,\bseta}\left(\bsg_{\{1 \mcol d-1\}},g_{d,m-2} + g \, p_{m-2} + \bar{g} \,p_{m-1} \right)
		\\
		&\quad=
		\frac{1}{b^2} \sum_{\bar{g} \in G_{b,2}} H_{d,m,\bseta}\left(\bsg_{\{1 \mcol d-1\}},g_{d,m-3} + g \, p_{m-3} + \bar{g} \,p_{m-2} \right),
	\end{align*}
	where again we split up $g_{d,m-2}= g_{d,m-3}+ g \, p_{m-3}$ according to Algorithm \ref{alg:cbcdbd}. Inductively repeating this argument and combining
	the result with the estimate in \eqref{eq:est_H_ind}, we obtain the inequality
	\begin{equation*}
		H_{d,m,\bseta}(\bsg)
		\le
		\frac{1}{b^{m-1}} \sum_{\bar{g} \in G_{b,m-1}} H_{d,m,\bseta} \left(\bsg_{\{1 \mcol d-1\}}, 1 + \bar{g} \,p_1 \right)
		,
	\end{equation*}
	where we used that in Algorithm \ref{alg:cbcdbd} we set $g_{d,1}=1$. Then, using Lemma \ref{lem:H-average} with $w=1$, $g_d=1$, and $g=0$ 
	to equate the right-hand side of the previous estimate, we finally obtain
	\begin{align*}
		H_{d,m,\bseta}(\bsg)
		&\le
		\sum_{t=1}^{m} \sum_{\substack{\ell=1 \\ \ell \not\equiv 0 \tpmod b }}^{b^t-1}
		\eta_d (1-b) b^{1-t} \left( \floor{\log_b \left(v_1 \left( \frac{\ell(x)}{x} \right) \right)} + 1 \right) \times
		\notag \\
		&\qquad\qquad\qquad\quad \times \prod_{j=1}^{d-1} \left(1 + \eta_j (1-b) \left( \floor{\log_b \left(v_t \left( \frac{\ell(x) \, g_j(x)}{x^t} \right) \right)} + 1 \right) \right) - (b^m-1)
		\\
		&\quad+
		\sum_{t=1}^m \!\! \sum_{\substack{\ell=1 \\ \ell \not\equiv 0 \tpmod b }}^{b^t-1} \!\!
		\prod_{j=1}^{d-1} \left(1 + \eta_j (1-b) \!\left(\floor{\log_b \left(\!v_t \left( \frac{\ell(x) \, g_j(x)}{x^t} \right) \right)} + 1 \right) \right) \!\left( 1 + \eta_d (1-b^{1-t}) \right)
		.
	\end{align*}
	For $\ell$ with $\ell \not\equiv 0 \tpmod b$, which is equivalent to $\gcd(\ell,p_1)=1$, we have for some $a \in \F_b \setminus \{0\}$ that 
	$\floor{\log_b \left(v_1 \left( \ell(x) / x \right) \right)} + 1 = \floor{\log_b \left( a / b \right)} + 1 = - 1 + 1 = 0$.
	Hence, we get
	\begin{align*}
		H_{d,m,\bseta}(\bsg)
		&\le
		\sum_{t=1}^m \sum_{\substack{\ell=1 \\ \ell \not\equiv 0 \tpmod b }}^{b^t-1}
		\prod_{j=1}^{d-1} \left(1 + \eta_j (1-b) \!\left(\floor{\log_b \left(\!v_t \left( \frac{\ell(x) \, g_j(x)}{x^t} \right) \right)} + 1 \right) \right) \times
		\\
		&\qquad\qquad\qquad\quad\times
		\left( 1 + \eta_d (1-b^{1-t}) \right) - (b^m-1)
		\\
		&\le
		\left( 1 + \eta_d \right) (H_{d-1,m,\bseta}(\bsg_{\{1:d-1\}}) + (b^m-1)) - (b^m-1)
		\\
		&=
		\left( 1 + \eta_d \right) H_{d-1,m,\bseta}(\bsg_{\{1:d-1\}}) + \eta_d (b^m-1)
		,
	\end{align*}
	which is the claimed estimate.
\end{proof}

Based on the result in Theorem \ref{thm:H_s_s-1} we can use an inductive argument to show that the quantity $H_{d,m,\bseta}(\bsg)$ is sufficiently small
if $\bsg$ has been constructed by Algorithm \ref{alg:cbcdbd}.

\begin{theorem} \label{thm:H_bound}
	Let $b$ be prime, let $m,d \in \N$ be positive integers and let $\bseta = (\eta_j)_{j\ge 1}$ be positive product weights.
	Then the generating vector $\bsg$ constructed by Algorithm \ref{alg:cbcdbd} satisfies 
	\begin{equation*}
		H_{d,m,\bseta}(\bsg)
		\le
		b^m \left[-1 + \prod_{j=1}^d ( 1 + \eta_j ) \right]
		.
	\end{equation*}
\end{theorem}

\begin{proof}
	Due to the formulation of Algorithm \ref{alg:cbcdbd}, the estimate \eqref{eq:H_s_s-1} obtained in Theorem \ref{thm:H_s_s-1} holds 
	if we replace $d$ by $r$ for any $r \in \{2,\ldots,d\}$, such that we get a result for $H_{r,m,\bseta}(\bsg)$ for any $r \in \{2,\ldots,d\}$. 
	Hence, we can use this estimate inductively to obtain
	\begin{align} \label{eq:estimate_H_ind}
		H_{d,m,\bseta}(\bsg)
		&\le
		( 1 + \eta_d ) H_{d-1,m,\bseta}(\bsg_{\{1:d-1\}}) + \eta_d (b^m-1)
		\notag \\
		&\le
		( 1 + \eta_d ) ( 1 + \eta_{d-1} ) H_{d-2,m,\bseta}(\bsg_{\{1:d-2\}}) + (1 + \eta_d) \eta_{d-1} (b^m-1) + \eta_d (b^m-1)
		\notag \\
		&=
		H_{d-2,m,\bseta}(\bsg_{\{1:d-2\}}) \prod_{j=d-1}^d ( 1 + \eta_j ) + (b^m - 1) \left[-1 + \prod_{j=d-1}^d ( 1 + \eta_j ) \right]
		\notag \\
		&\le
		H_{1,m,\bseta}(g_1) \prod_{j=2}^d ( 1 + \eta_j ) + (b^m - 1) \left[-1 + \prod_{j=2}^d ( 1 + \eta_j ) \right]
		.
	\end{align}
	Next, we observe that
	\begin{align*}
		H_{1,m,\bseta}(g_1)
		&=
		H_{1,m,\bseta}(1)
		\\
		&=
		\sum_{n=1}^{b^m-1} \left(1 + \eta_1 (1-b) \left( \floor{\log_b \left(v_m \left( \frac{n(x)}{x^m} \right) \right)} + 1 \right) \right) - (b^m-1)
		\\
		&=
		- \eta_1 \sum_{n=1}^{b^m-1} (b-1) \left( \floor{\log_b \left(v_m \left( \frac{n(x)}{x^m} \right) \right)} + 1 \right)
		\\
		&=
		- \eta_1 \sum_{t=1}^{m} \sum_{\substack{n=1 \\ n \not\equiv 0 \tpmod b }}^{b^t-1} (b-1) \floor{\log_b \left(v_t \left( \frac{n(x)}{x^t} \right) \right)}
		+ \eta_1 (1-b) (b^m-1)
		\\
		&=
		- \eta_1 \sum_{t=1}^{m} \sum_{r=0}^{t-1} \sum_{\substack{n=1 \\ n \not\equiv 0 \tpmod b \\ \deg(n(x))=r }}^{b^t-1} (b-1) 
		\floor{\log_b \left(v_t \left( \frac{n(x)}{x^t} \right) \right)} + \eta_1 (1-b) (b^m-1).
	\end{align*}
	For any polynomial $n(x) \in \F_b[x]$ of degree $0 \le r < t$ with $\gcd(n,x)=1$, we have  that
	\begin{equation*}
		\floor{\log_b\left(v_t\left( \frac{n(x)}{x^t} \right) \right)}
		=
		-(t - r)
	\end{equation*}
	such that we can further deduce that
	\begin{align*}
		H_{1,m,\bseta}(g_1)
		&=
		\eta_1 \sum_{t=1}^{m} (b-1) \sum_{r=0}^{t-1} \sum_{\substack{n=1 \\ n \not\equiv 0 \tpmod b \\ \deg(n(x))=r }}^{b^t-1} (t-r) + \eta_1 (1-b) (b^m-1)
		\\
		&=
		\eta_1 \sum_{t=1}^{m} (b-1) \left( (b-1) t + \sum_{r=1}^{t-1} (b-1)^2 b^{r-1} (t-r) \right) + \eta_1 (1-b) (b^m-1)
		\\
		&=
		\eta_1 \sum_{t=1}^{m} (b-1) \left( (b-1) t + b^t - bt + t - 1 \right) + \eta_1 (1-b) (b^m-1)
		\\
		&=
		\eta_1 (b-1) \sum_{t=1}^{m} (b^t  - 1) + \eta_1 (1-b) (b^m-1)
		\\
		&=
		\eta_1 (b^{m+1} - bm - b + m) + \eta_1 (1-b) (b^m-1)
		=
		\eta_1 (b^m - (b-1)m - 1)
		.
	\end{align*}
	Combining this with the estimate in \eqref{eq:estimate_H_ind}, we finally obtain
	\begin{align*}
		H_{d,m,\bseta}(\bsg)
		&\le
		\eta_1 \,(b^m - 1) \prod_{j=2}^d ( 1 + \eta_j ) + (b^m - 1) \left[-1 + \prod_{j=2}^d ( 1 + \eta_j ) \right]\\
		&=
		(b^m - 1) \left[-1 + \prod_{j=1}^d ( 1 + \eta_j ) \right] ,
	\end{align*}
	which yields the claimed estimate.
\end{proof}
\noindent
Theorem \ref{thm:H_bound} allows us to prove the following result regarding the construction in Algorithm \ref{alg:cbcdbd}.

\begin{theorem} \label{thm:cbcdbd-T}
	Let $b$ be prime, let $m,d \in \N$ with $m\ge 4$, and let $(\eta_j)_{j\ge1}$ be positive product weights.
	Then the generating vector $\bsg$ constructed by Algorithm \ref{alg:cbcdbd} satisfies 
	\begin{equation*}
		T_{\bseta} (\bsg,p_m) 
		\le 
		\frac{1}{b^m} \left[ \prod_{j=1}^d (1 + \eta_j ((b-1)m+1)) + b \,m \prod_{j=1}^d \left(1 + \eta_j \left(2(b-1)m + \frac{2b}{b-1}\right) \right) \right]
		.
	\end{equation*}
\end{theorem}

\begin{proof}
	We remark that for reals $a_1,\ldots,a_d \in \R$ the general identity
	\begin{equation*}
		\sum_{\emptyset\ne \setu \subseteq \{1:d\}} \prod_{j \in \setu} a_j
		=
		-1 + \prod_{j=1}^d (1 + a_j)
	\end{equation*}
	holds. Using the bound on $T_{\bseta}(\bsg,p_m)$ in Theorem \ref{thm:CBCDBD_split} and inserting for $\bsg$ the generating vector
	obtained from Algorithm \ref{alg:cbcdbd}, for which the bound on $H_{d,m,\bseta}(\bsg)$ from Theorem \ref{thm:H_bound} 
	holds, yields
	\begin{align*}
		T_{\bseta}(\bsg,p_m) 
		&\le 
		\left[-1 + \prod_{j=1}^d ( 1 + \eta_j ) \right] - \left[-1 + \prod_{j=1}^d ( 1 + \eta_j ) \right]
		+ \sum_{\emptyset\ne \setu \subseteq \{1:d\}} \frac{\eta_\setu}{b^m} ((b-1)m +1)^{|\setu|}
		\\ 
		&\quad+ 
		\sum_{\emptyset\ne \setu \subseteq \{1:d\}} \frac{\eta_\setu}{b^m} \, (b \,m \abs{\setu})
		\left((b-1)m + \frac{b}{b-1} \right)^{\abs{\setu}}
		\\
		&\le \frac{1}{b^m} \left[ \prod_{j=1}^d (1 + \eta_j ((b-1)m+1)) 
		+ b \,m \prod_{j=1}^d \left(1 + \eta_j \left(2(b-1)m + \frac{2b}{b-1}\right) \right) \right]
		,
	\end{align*}
	where in the last step we used that $\abs{\setu} \le 2^{\abs{\setu}}$. Note that by the formulation of Algorithm \ref{alg:cbcdbd} we have that 
	$\gcd(g_j,p_m)=1$ for $1 \le j \le d$ such that the conditions of Theorem \ref{thm:CBCDBD_split} are satisfied.
\end{proof}

The next theorem states the main result of this paper, implying that by the construction in Algorithm \ref{alg:cbcdbd} we obtain an error convergence rate that is arbitrarily close to the optimal rate of $N^{-\alpha}$ (we know that this order is optimal due to the relation between the worst-case errors in $W_{d,\bsgamma}^\alpha$ and 
$\widetilde{W}_{d,\bsgamma}^\alpha$ stated in Section \ref{sec:walsh-polylat} and due to the fact that the rate $N^{-\alpha /2}$ is optimal in $\widetilde{W}_{d,\bsgamma}^\alpha$). Additionally, under a summability condition on the weights that is common in the related literature, the error can be bounded independently of the dimension, by which we obtain what is known as strong polynomial tractability in the context of information-based complexity.

\begin{theorem}\label{thm:main}
	Let $b$ be prime, let $m,d \in \N$ with $m\ge 4$, let $N=b^m$, and let $(\gamma_j)_{j\ge1}$ be positive product weights satisfying
	\begin{equation*}
		\sum_{j \ge 1} \gamma_j 
		<
		\infty
		.
	\end{equation*}
	Furthermore, denote by $\bsg$ the generating vector obtained by Algorithm \ref{alg:cbcdbd}, run for the weight sequence $\bseta =\bsgamma=(\gamma_j)_{j\ge 1}$.	
	Then, for any $\delta>0$ and each $\alpha>1$, the generating vector $\bsg$ satisfies
	\begin{equation*}
		e_{b^m,d,\alpha,\bsgamma^\alpha}(\bsg) 
		\le 
		\frac{1}{N^\alpha} \left( C(\bsgamma^\alpha) + \bar{C}\left(\bsgamma,\delta\right) \, N^{\alpha \delta} \right)
		,
	\end{equation*}
	with positive constants $C(\bsgamma^\alpha)$ and $\bar{C}\left(\bsgamma,\delta\right)$, which are independent of $d$ and $N$.
	
	Additionally, if Algorithm \ref{alg:cbcdbd} is run for the weights $\bseta=\bsgamma^{1/\alpha}$ with $\alpha > 1$, which satisfy
	\begin{equation*}
		\sum_{j \ge 1} \gamma_j^{1/\alpha} 
		<
		\infty
		,
	\end{equation*}	
	then, for any $\delta>0$, the resulting generating vector $\widetilde{\bsg}$ satisfies the error bound 
	\begin{equation*}
		e_{b^m,d,\alpha,\bsgamma}(\widetilde{\bsg})
		\le 
		\frac{1}{N^\alpha} \left( K(\bsgamma) + \bar{K}(\bsgamma^{1/\alpha},\delta) \, N^{\alpha \delta} \right)
		,
	\end{equation*}
	with positive constants $K(\bsgamma)$ and $\bar{K}(\bsgamma^{1/\alpha},\delta)$, which are independent of $d$ and $N$.
\end{theorem}

\begin{proof}
	We know from Proposition \ref{prop:trunc_error} that
	\begin{align*}
		e_{b^m,d,\alpha,\bseta^\alpha}(\bsg) 
		\le
		\frac{1}{N^{\alpha}} \sum_{\emptyset\neq \setu \subseteq \{1:d\}} \eta^\alpha_{\setu} (2 \mu_b(\alpha))^{\abs{\setu}} 
		+
		T_{\alpha,\bseta^\alpha}(\bsg,p_m)
		.
	\end{align*}
	For the special case of product weights $\eta_{\setu}=\prod_{j\in\setu} \eta_j$, $u\subseteq \{1 \mcol d\}$, this yields
	\begin{align*}
		e_{b^m,d,\alpha,\bseta^\alpha}(\bsg)
		\le 	
		\frac{1}{N^{\alpha}} \prod_{j=1}^d \left(1 + 2\mu_b (\alpha) \eta_j^\alpha \right) 
		+
		T_{\alpha,\bseta^\alpha}(\bsg,p_m)
		.
	\end{align*}
	Since $\alpha>1$, we can use an inequality, sometimes referred to as Jensen's inequality, which states that 
	$\sum_{i=1}^M y_i \le \left(\sum_{i=1}^M y^p_i \right)^{1/p}$ for non-negative $y_1,\ldots,y_M$ and $0 \le p \le 1$. This yields
	\begin{align*}
		 T_{\alpha,\bseta^\alpha}(\bsg,p_m)
		 &=
		 \sum_{\bszero \ne \bsk \in A_p(\bsg)} (r_{\alpha,\bseta^\alpha}(\bsk))^{-1}
		 =
		 \sum_{\bszero \ne \bsk \in A_p(\bsg)} (r_{1,\bseta}(\bsk))^{-\alpha}
		 \\
		 &\le
		 \left(\sum_{\bszero \ne \bsk \in A_p(\bsg)} (r_{1,\bseta}(\bsk))^{-1}\right)^\alpha
		 =
		 \left(T_{\bseta}(\bsg,p_m)\right)^\alpha,
	\end{align*}
	and by Theorem \ref{thm:cbcdbd-T} we know that Algorithm \ref{alg:cbcdbd} run for weights $\bseta$ yields $\bsg$ which satisfy
	\begin{equation*}
		T_{\bseta}(\bsg,p_m) 
		\le 
		\frac{1}{b^m} \left[ \prod_{j=1}^d (1 + \eta_j ((b-1)m+1)) + b \,m \prod_{j=1}^d \left(1 + \eta_j  
		\left(2(b-1)m + \frac{2b}{b-1}\right) \right) \right]
		.
	\end{equation*}
	From this, we deduce, using either the weights $\bseta=\bsgamma^{1/\alpha}$ or $\bseta=\bsgamma$ for Algorithm \ref{alg:cbcdbd}, that
	\begin{align*}
		b^m \,T_{\bseta}(\bsg,p_m) 
		&\le
		\prod_{j=1}^d (1 + \eta_j ((b-1)m+1)) + b \,m \prod_{j=1}^d \left(1 + \eta_j  \left(2(b-1)m + \frac{2b}{b-1}\right) \right)
		\\
		&\le
		\prod_{j=1}^d (1 + \eta_j 4 b m) + b \,m \prod_{j=1}^d (1 + \eta_j 4 b m)
		=
		(1 + b \, m) \prod_{j=1}^d (1 + \eta_j 4 b m)
		\\
		&\le
		\widetilde{C}(\delta/2) \, b^{m \delta / 2} \prod_{j=1}^d (1 + \eta_j 4 b m)
		\le
		\widetilde{C}(\delta/2) \, b^{m \delta / 2} \prod_{j=1}^\infty (1 + \eta_j 4 b m)
	\end{align*}
	for arbitrary $\delta>0$, where $\widetilde{C}(\delta/2)$ is a constant depending only on $\delta$. Due to the imposed condition on the weights, i.e.,
	$\sum_{j \ge 1} \gamma_j < \infty$ or $\sum_{j \ge 1} \gamma_j^{1/\alpha} < \infty$, we can use the result in \cite[Lemma 3]{HN03} to see that the 
	last product can be bounded by $\widehat{C} (\bsgamma) b^{m \delta / 2}$ or $\widehat{C} (\bsgamma^{1/\alpha}) b^{m \delta / 2}$, respectively,
	where $\widehat{C} (\bsgamma)$ and $\widehat{C} (\bsgamma^{1/\alpha})$ may depend on the weights $\bsgamma$ or $\bsgamma^{1/\alpha}$, but are independent of the dimension. Choosing $\bseta=\bsgamma$, this yields that
	\begin{equation*}
		\left(T_{\bseta}(\bsg,p_m)\right)^\alpha
		=
		\left(T_{\bsgamma}(\bsg,p_m)\right)^\alpha 
		\le
		\frac{1}{N^\alpha}\left(\widetilde{C}(\delta/2)\right)^\alpha \left(\widehat{C} (\bsgamma)\right)^{\alpha} N^{\alpha\delta}
		,
	\end{equation*}
	and similarly, for $\bseta=\bsgamma^{1/\alpha}$,
	\begin{equation*}
		\left(T_{\bseta}(\bsg,p_m)\right)^\alpha
		=
		\left(T_{\bsgamma^{1/\alpha}}(\bsg,p_m)\right)^\alpha 
		\le 
		\frac{1}{b^{m\alpha}}\left(\widetilde{C}(\delta/2)\right)^\alpha \left(\widehat{C} (\bsgamma^{1/\alpha})\right)^\alpha N^{\alpha\delta}
		.
	\end{equation*}
	Setting then $C(\bsgamma^\alpha)=\prod_{j=1}^d (1 + 2\mu_b (\alpha) \gamma_j^\alpha)$ and 
	$\bar{C}(\bsgamma,\delta) = (\widetilde{C}(\delta/2))^\alpha \left(\widehat{C} (\bsgamma)\right)^\alpha$,
	and, similarly, $K(\bsgamma)=\prod_{j=1}^d (1 + 2\mu_b (\alpha) \gamma_j)$ and 
	$\bar{K}(\bsgamma^{1/\alpha},\delta) = (\widetilde{C}(\delta/2))^\alpha \left(\widehat{C} (\bsgamma^{1/\alpha})\right)^\alpha$, 
	we obtain the claimed error estimates, where the first stated bound holds simultaneously for all $\alpha > 1$.
\end{proof}

The result in Theorem \ref{thm:main} consists of two statements regarding the worst-case error behavior of generating vectors constructed by Algorithm \ref{alg:cbcdbd}.
On the one hand, when run with weights $\bsgamma^{1/\alpha}$, and hence depending on the parameter $\alpha$, the algorithm yields typical error bounds for the
worst-case error in the space $W_{d,\bsgamma}^\alpha$. We emphasize that this type of result could also be obtained by formulating and using an analogous CBC-DBD algorithm which is instead directly based on the search criterion $e_{b^m,d,\alpha,\bsgamma}$. On the other hand, when run with weights $\bsgamma$, thus independently of $\alpha$, the algorithm produces generating vectors for which bounds on the worst-case errors in the spaces $W_{d,\bsgamma^\alpha}^\alpha$ hold simultaneously for all $\alpha > 1$.

\section{Fast implementation of the construction scheme} \label{sec:fast_impl}

In this section we discuss the efficient implementation of the introduced CBC-DBD algorithm and analyze its complexity. Throughout this section, we will consider the implementation for the special case of $b=2$ and product weights $\gamma_{\setu} = \prod_{j\in\setu} \gamma_j$ for a sequence of positive reals $(\gamma_j)_{j\ge1}$. Choosing the prime base as $b=2$ allows for the use of bitwise operations which facilitate an efficient implementation of the construction scheme. We remark that the major challenge for the implementation of the algorithm for $b>2$ is an efficient computation of the polynomial multiplication modulo $b$, all other steps of the algorithm can be implemented analogously.

\subsection{Implementation and cost analysis of the CBC-DBD algorithm}

Let $q \in \F_2[x]$, $m,d \in \N$ be positive integers and let $\bsgamma = (\gamma_{\setu})_{\setu \subseteq \{1:d\}}$, where $\gamma_\setu=\prod_{j \in \setu} \gamma_j$ 
with positive reals $(\gamma_j)_{j\ge1}$. We recall that for $b=2$ and integers $w \in \{1 \mcol m\}$, $r \in \{1 \mcol d\}$ the digit-wise quality function $h_{r,w,m,\bsgamma}$
in Definition \ref{def:h_r,w}, which is used in Algorithm \ref{alg:cbcdbd}, is given by
\begin{align*}
	h_{r,w,m,\bsgamma}(q)
	&=
	\sum_{t=w}^m \frac{1}{2^{t-w}} \sum_{\substack{\ell=1 \\ \ell \equiv 1 \tpmod 2}}^{2^t-1}
	\left( 1 - \gamma_r \left( \floor{\log_2 \left(v_w \left( \frac{\ell(x) \, q(x)}{x^w} \right) \right)} + 1 \right) \right) \times
	\notag \\
	&\quad\times \prod_{j=1}^{r-1} \left(1 - \gamma_j \left( \floor{\log_2 \left(v_t \left( \frac{\ell(x) \, g_j(x)}{x^t} \right) \right)} + 1 \right) \right)
	,
\end{align*}
where the polynomials $g_1,\ldots,g_{r-1} \in \F_2[x]$ have been determined in the previous steps of the algorithm. Since the cost of a single evaluation of the function
$h_{r,w,m,\bsgamma}$ is crucial for the total cost of Algorithm \ref{alg:cbcdbd}, we are interested in an efficient evaluation procedure which will be discussed in the following paragraph.

For integers $t \in \{2,\ldots,m\}$ and odd $\ell \in \{1,\ldots,2^t - 1\}$, we define the term $a(r, t, \ell)$ as
\begin{equation*}
	a(r,t,\ell)
	:=
	\prod_{j=1}^r \left(1 - \gamma_j \left( \floor{\log_2 \left(v_t \left( \frac{\ell(x) \, g_j(x)}{x^t} \right) \right)} + 1 \right) \right)
\end{equation*}
and observe that for the evaluation of $h_{r,w,m,\bsgamma}(q)$ we can compute and store the term $a(r-1,t,\ell)$ since it is independent of $w$ and $q$.
This way we can rewrite $h_{r,w,m,\bsgamma}(q)$ as
\begin{equation} \label{eq:fast-eval-h_rw}
	h_{r,w,m,\bsgamma}(q)
	=
	\sum_{t=w}^m \frac{1}{2^{t-w}} \sum_{\substack{\ell=1 \\ \ell \equiv 1 \tpmod 2}}^{2^t-1}
	a(r-1,t,\ell) \left( 1 - \gamma_r \left( \floor{\log_2 \left(v_w \left( \frac{\ell(x) \, q(x)}{x^w} \right) \right)} + 1 \right) \right)
	,
\end{equation}
where in Algorithm \ref{alg:cbcdbd}, after having determined $g_{r,w}$, the values of $a(r,w,\ell)$ for odd integers ${\ell\in\{1,\ldots,2^w-1\}}$ 
are computed via the recurrence relation
\begin{equation*}
	a(r,w,\ell) 
	=
	a(r-1,w,\ell) \left( 1 - \gamma_r \left( \floor{\log_2 \left(v_w \left( \frac{\ell(x) \, g_{r,w}(x)}{x^w} \right) \right)} + 1 \right) \right)
	.
\end{equation*}

For an algorithmic implementation, we introduce the vector $\bsv=(v(1),\ldots,v(2^m-1)) \in \R^{2^m-1}$ whose components, for the current $r \in \{1,\ldots,d\}$, are given by 
\begin{equation*}
	v(\ell \, 2^{m-t}) 
	= 
	\prod_{j=1}^r \left( 1 - \gamma_j \left( \floor{\log_2 \left(v_t \left( \frac{\ell(x) \, g_j(x)}{x^t} \right) \right)} + 1 \right) \right)
	=
	a(r,t,\ell)
\end{equation*}
for each $t=1,\ldots,m$ and corresponding odd index $\ell\in\{1,\ldots,2^t - 1\}$. Furthermore, note that for the evaluation of $h_{r,w,m,\bsgamma}$ we do not 
require the values of $a(r,t,\ell)$ for $t=2,\ldots,w-1$. Combining these findings leads to the following fast implementation of Algorithm \ref{alg:cbcdbd}.

\begin{algorithm}[H]
	\caption{Fast component-by-component digit-by-digit algorithm}
	\label{alg:fast-cbcdbd}
	\vspace{5pt}
	\textbf{Input:} Integers $m, d \in \N$ and positive weights $(\gamma_j)_{j=1}^d$. \\
	\vspace{-10pt}
	\begin{algorithmic}
		\FOR{$\ell=1$ \TO $2^m-1$}
		\STATE $v(\ell) = 1 - \gamma_1 \left( \floor{\log_2 \left(v_m \left( \frac{\ell}{x^m} \right) \right)} + 1 \right)$
		\ENDFOR
		\vspace{5pt}
		\STATE Set $g_{1,m} = 1$ and $g_{2,1} = \cdots = g_{d,1} = 1$.
		\vspace{5pt}
		\FOR{$r=2$ \TO $d$}
		\FOR{$w=2$ \TO $m$}
		\STATE $g^{\ast} = \underset{g \in \F_2}{\argmin} \;  h_{r,w,m,\bsgamma}(g_{r,w-1} + g \,x^{w-1})$ 
		with $h_{r,w,m,\bsgamma}$ evaluated using \eqref{eq:fast-eval-h_rw}
		\STATE $g_{r,w} = g_{r,w-1} + g^{\ast} x^{w-1}$
		\FOR{$\ell=1$ \TO $2^w-1$ \textbf{in steps of} $2$}
		\STATE $v(\ell \, 2^{m-w}) = v(\ell \, 2^{m-w}) \left( 1 - \gamma_r \left( \floor{\log_2 \left(v_w \left( \frac{\ell \, g_{r,w}}{x^w} \right) \right)} + 1 \right) \right)$
		\ENDFOR
		\ENDFOR
		\ENDFOR
		\vspace{5pt}
		\STATE Set $\bsg = (g_1,\ldots,g_d)$ with $g_r := g_{r,m}$ for $r=1,\ldots,d$.
	\end{algorithmic}
	\vspace{5pt}
	\textbf{Return:} Generating vector $\bsg = (g_1,\ldots,g_d) \in (G_{2,m}^\ast)^d$ for $N=2^m$.
\end{algorithm}

The computational complexity of Algorithm \ref{alg:fast-cbcdbd} is then summarized in the following theorem.

\begin{theorem} \label{thm:cost-cbcdbd}
	Let $m,d \in \N$ and let $\bsgamma = (\gamma_j)_{j=1}^d$ be a given sequence of positive weights. Then Algorithm \ref{alg:fast-cbcdbd} constructs a generating  
	vector $\bsg=(g_1,\ldots,g_d) \in (G_{2,m}^\ast)^d$ using $\calO(d \,m\, 2^m)$ operations and requiring $\calO(2^m)$ memory.
\end{theorem}

\begin{proof} 
	Due to the relation in \eqref{eq:fast-eval-h_rw}, the cost of evaluating $h_{r,w,m,\bsgamma}(q)$ can be reduced to $\calO(\sum_{t=w}^m 2^{t-1})$ operations.
	Thus, the number of calculations in the inner loop over $w = 2,\dots,m$ of Algorithm \ref{alg:fast-cbcdbd} is of order
	\begin{equation*}
	\calO\left(\sum_{w=2}^m 2 \sum_{t=w}^m 2^{t-1}\right) 
	=
	\calO\left(\sum_{w=2}^m \sum_{t=w}^m 2^t \right)
	=
	\calO\left(m \, 2^m - 2(2^m - 1) \right)
	=
	\calO\left(m \, 2^m \right)
	.
	\end{equation*}
	Hence, the outer loop over $r=2,\ldots,d$, which is the main cost of Algorithm \ref{alg:fast-cbcdbd}, can be executed in $\calO\left(d \,m\, 2^m \right)$ operations. 
	Furthermore, we observe that initialization and updating of the vector $\bsv \in \R^{2^m-1}$ can both be executed in $\calO(2^m)$ operations. 
	Additionally, storing the vector $\bsv$ requires $\calO(2^m)$ of memory.
\end{proof}

We remark that the running time of Algorithm \ref{alg:fast-cbcdbd} can be reduced further by precomputing and storing the $2^m-1$ values 
\begin{equation*}
\left( \floor{\log_2 \left(v_m \left( \frac{\ell}{x^m} \right) \right)} + 1 \right)
\quad \text{for} \quad
\ell=1,\ldots,2^m-1 
.
\end{equation*}

The derivation leading to the fast implementation in Algorithm \ref{alg:fast-cbcdbd} is using arguments that were used in \cite{EKNO2020}, where
a component-by-component digit-by-digit construction for lattice rules in weighted Korobov spaces has been studied. 
Theorem \ref{thm:cost-cbcdbd} shows that the fast implementation of the component-by-component digit-by-digit construction for 
polynomial lattice rules achieves the same computational complexity as state-of-the-art component-by-component methods, see, e.g., 
\cite{DKS13}. In these constructions the speed-up of the algorithm is achieved by reordering the involved matrices to be of circulant structure
and by then employing a fast matrix-vector product which uses fast Fourier transformations (FFTs). We refer to 
\cite{NC06a} for further details on an implementation for polynomial lattice rules. In contrast, our method does not rely on the use of 
FFTs and the low time complexity of the resulting algorithm is due to the smaller search space for the components $g_j$ of the generating vector 
$\bsg$. Furthermore, we remark that the mentioned state-of-the-art CBC constructions mainly use a primitive or irreducible modulus 
$p \in \F_2[x]$ since then the multiplicative group of $\F_2[x] / (p)$ is cyclic. While for reducible polynomials, such as $p(x)=x^m$, 
a fast CBC construction is theoretically possible by using a similar strategy as for the fast CBC construction for lattice rules with a 
composite number of points, there are, to the best of our knowledge, no explicit implementations of such an algorithm known. 
On the other hand, the CBC-DBD construction considered in this article immediately yields a fast algorithm for 
the construction of polynomial lattice rules in $\calO(d \,m\, 2^m)$ operations for $p(x)=x^m$.

\section{Numerical results} \label{sec:num}

In this section, we illustrate the error convergence behavior of the polynomial lattice rules constructed by the CBC-DBD algorithm and visualize the computational complexity of the construction by means of numerical experiments. As in the previous section, we consider polynomial lattice rules in the weighted Walsh space $W_{d,\bsgamma}^\alpha$ for prime base $b=2$ and product weights $\gamma_\setu = \prod_{j \in \setu} \gamma_j$ given in terms of positive reals 
$(\gamma_j)_{j \ge 1}$. \\

In order to demonstrate the performance of the algorithm, we compare the worst-case errors of the constructed polynomial lattice rules as well as the algorithm's computation times to the corresponding quantities obtained by a state-of-the-art component-by-component algorithm, see, e.g., \cite{DKS13}. As remarked in the previous section, no fast CBC construction is known for the case $p(x)=x^m$ such that instead we compare our algorithm with a CBC construction with primitive polynomial $p \in \F_2[x]$ of degree $m$ as the modulus. Both constructions deliver polynomial lattice rules for the spaces $W_{d,\bsgamma}^\alpha$ consisting of $2^m$ cubature points.

The different algorithms have been implemented in MATLAB R2019b and Python 3.6.3. In Python the implementations are available in double-precision as well as arbitrary-precision floating-point arithmetic with the latter provided by the multiprecision Python library mpmath.    

\subsection{Error convergence behavior}

Let $m,d \in \N$, $\alpha > 1$, and a sequence of positive weights $\bsgamma = (\gamma_j)_{j \ge 1}$ be given. By Theorem \ref{thm:wce_dig_net}, the worst-case 
error of a polynomial lattice point set $P(\bsg,p)=\{\bsx_0,\ldots,\bsx_{b^m-1}\}$ in base $b=2$ with generating vector $\bsg$ and modulus $p \in \F_2[x]$, with $\deg(p)=m$, in the space $W_{d,\bsgamma}^\alpha$ is given by
\begin{align*}
	e_{b^m,d,\alpha,\bsgamma}(\bsg)
	&=
	\sum_{\bszero \ne \bsk \in \calD(\bsg,p)} \left( r_{\alpha,\bsgamma}(\bsk) \right)^{-1}
	=
	\frac{1}{b^m} \sum_{n=0}^{b^m-1} \sum_{\bszero \ne \bsk \in \N_0^d} \gamma_{\supp(\bsk)} \frac{\wal_{\bsk} (\bsx_n)}{r_{\alpha}(\bsk)}
	.
\end{align*}
For $b=2$ and product weights $\gamma_{\setu} = \prod_{j \in \setu} \gamma_j$, this expression then equals
\begin{equation*}
	e_{2^m,d,\alpha,\bsgamma}(\bsg)
	=
	-1 + \frac{1}{2^m} \sum_{n=0}^{2^m-1} \prod_{j=1}^d \left(1 + \gamma_j \,\phi_\alpha(x_{n,j}) \right)
\end{equation*}
with $\phi_\alpha: [0,1] \to \R$ given by
\begin{equation*}
	\phi_\alpha(x)
	=
	\left\{\begin{array}{ll}
	\mu_2(\alpha), & {\text{if }} x=0 , \\ 
	\mu_2(\alpha) - 2^{(1+t)(\alpha-1)} (\mu_2(\alpha)+1), & {\text{otherwise, with }} t=\floor{\log_2(x)} ,
	\end{array}\right.
\end{equation*}
see, e.g., \cite{DP05}. For the polynomial lattice rules constructed by the algorithms considered, we will use this worst-case error expression as a measure of quality. 

In particular, we consider the convergence behavior of the worst-case error $e_{2^m,d,\alpha,\bsgamma^\alpha}(\bsg)$ for generating vectors $\bsg$ obtained by the CBC-DBD algorithm (with modulus $p(x)=x^m$) and compare it with the error rates for polynomial lattice rules constructed by the standard fast CBC algorithm (with primitive polynomial $p \in \F_2[x]$ of degree $m$) which uses the worst-case error $e_{2^m,d,\alpha,\bsgamma^\alpha}$ as the quality criterion. We display the computation results for dimension $d=100$ for different sequences of product weights $\bsgamma = (\gamma_j)_{j\ge1}$, different values of $m$, and different smoothness parameters $\alpha$. We stress that the almost optimal error rates of $\calO(N^{-\alpha+\delta})$, guaranteed by Theorem \ref{thm:main}, may not always be visible for the weights and ranges of $N$ considered in our numerical experiments. The graphs shown are therefore to be understood as an illustration of the pre-asymptotic behavior of the worst-case error.

\begin{remark}
	We stress that in these numerical experiments we compare the CBC-DBD algorithm with modulus $p(x)=x^m$ to the CBC construction with a primitive modulus polynomial. Both constructions yield polynomial lattices consisting of $N=b^m$ points that have been constructed for the same function space $W_{d,\bsgamma}^\alpha$ such that the comparison is valid. To the best of our knowledge, there is no known implementation of the fast CBC algorithm for polynomial lattice rules based on the modulus $p(x)=x^m$. The reason for the elusiveness of such an implementation is the more involved structure of the group of units of the factor ring $\F_b[x]/(x^m)$ when factored into cyclic groups, see, e.g.,
	\cite{SG85}. While for lattice rules the group of integer units modulo $N=b^m$ is either cyclic (for odd $b$) or can be factored into two cyclic subgroups (for $b=2$), which makes the corresponding generator easily computable, see, e.g., \cite{NC06a}, the ring $\F_b[x]/(x^m)$ factors into a larger number of cyclic subgroups (for sufficiently large $m$) and their generating elements are less studied in the context of QMC methods.
\end{remark}

The results in Figure \ref{fig:cbc-dbd} show that the CBC-DBD algorithm constructs generating vectors of good polynomial lattice rules which have worst-case errors that are comparable to those of polynomial lattice rules obtained by the fast CBC algorithm. We observe identical asymptotic error rates for both algorithms considered, and also
note that the CBC-DBD construction always delivers slightly higher error values. The latter behavior can easily be explained by the fact that the CBC construction is directly tailored to the space $W_{d,\bsgamma^\alpha}^{\alpha}$ for a particular $\alpha$ since $e_{b^m,d,\alpha,\bsgamma^\alpha}$ is used as the quality measure. In contrast, the CBC-DBD construction is independent of the smoothness parameter $\alpha$ and constructs polynomial lattices which have a good quality for all $\alpha > 1$. This in turn also means that the CBC-DBD algorithm only needs to be executed once while the CBC construction has to be run for all considered $\alpha$. Additionally, we observe that the pre-asymptotic error decay is determined by the weight sequence $\bsgamma = (\gamma_j)_{j\ge1}$. The faster the weights $\gamma_j$ decay, the closer the error rate is to the optimal rate of $\calO(N^{-\alpha})$ for the space $W_{d,\bsgamma^\alpha}^{\alpha}$.

\newpage

\begin{figure}[H]
	\centering
	\textbf{Error convergence in the space $W_{d,\bsgamma^\alpha}^{\alpha}$ with $d=100, \alpha=1.5,2,3$.} \par\medskip 
	\hspace{-0.25cm}
	\centering
	\begin{subfigure}[b]{0.5\textwidth}
		\centering
		\begin{tikzpicture}
			\begin{axis}[%
			width=0.8\textwidth,
			height=0.8\textwidth,
			at={(0\textwidth,0\textwidth)},
			scale only axis,
			xmode=log,
			xmin=42.6666666666667,
			xmax=98304,
			xminorticks=true,
			xlabel style={font=\color{white!15!black}},
			xlabel={Number of points $N=2^m$},
			ymode=log,
			ymin=1e-14,
			ymax=1,
			yminorticks=true,
			ylabel style={font=\color{white!15!black}},
			ylabel={Worst-case error $e_{N,d,\alpha,\mathbf{\gamma^{\alpha}}}(\bsg)$},
			axis background/.style={fill=white},
			axis x line*=bottom,
			axis y line*=left,
			xmajorgrids,
			xminorgrids,
			ymajorgrids,
			yminorgrids,
			minor grid style={opacity=0},
			legend style={at={(0.03,0.03)}, anchor=south west, legend cell align=left, align=left, draw=white!15!black}
			]
			\addplot [color=mycolor1, line width=0.9pt, mark=o, mark options={solid, mycolor1}, forget plot]
			table[row sep=crcr]{%
				64	0.0390712472682382\\
				128	0.0160728287242077\\
				256	0.0068443006743524\\
				512	0.0028984531538587\\
				1024	0.00115975507813964\\
				2048	0.000520313806071635\\
				4096	0.000185139231596003\\
				8192	7.63203299667065e-05\\
				16384	3.20499837535847e-05\\
				32768	1.37950245039247e-05\\
				65536	5.46684649549108e-06\\
			};
			\addplot [color=mycolor2, line width=0.9pt, mark=triangle, mark options={solid, mycolor2}, forget plot]
			table[row sep=crcr]{%
				64	0.0376151186365635\\
				128	0.0152933248325304\\
				256	0.00618948770005742\\
				512	0.00247447287359493\\
				1024	0.0009979148524534\\
				2048	0.000395689821292074\\
				4096	0.00015735848167537\\
				8192	6.27210633415442e-05\\
				16384	2.49346925597582e-05\\
				32768	9.8131128921238e-06\\
				65536	3.86074203669712e-06\\
			};
			\addplot [color=mycolor3, dashed, line width=0.9pt]
			table[row sep=crcr]{%
				64	0.0781424945364764\\
				128	0.0321717904637741\\
				256	0.0132453424706303\\
				512	0.00545319656243037\\
				1024	0.00224511769434734\\
				2048	0.000924330051881541\\
				4096	0.000380552897944936\\
				8192	0.000156676187082197\\
				16384	6.45046397785344e-05\\
				32768	2.65569939532393e-05\\
				65536	1.09336929909821e-05\\
			};
			\addlegendentry{\scriptsize $\mathcal{O}(N^{-1.28})$}
			
			\addplot [color=mycolor4, line width=0.9pt, mark=o, mark options={solid, mycolor4}, forget plot]
			table[row sep=crcr]{%
				64	0.00173189414099363\\
				128	0.000497893995938344\\
				256	0.000167316619582138\\
				512	5.72477209622562e-05\\
				1024	1.51301729566877e-05\\
				2048	6.12795467448452e-06\\
				4096	1.16678203187784e-06\\
				8192	3.52440735216716e-07\\
				16384	1.21191211480429e-07\\
				32768	4.6743013027553e-08\\
				65536	1.16138439342712e-08\\
			};
			\addplot [color=mycolor5, line width=0.9pt, mark=triangle, mark options={solid, mycolor5}, forget plot]
			table[row sep=crcr]{%
				64	0.00154678921195924\\
				128	0.00044006393769157\\
				256	0.000122429006985678\\
				512	3.45517665647055e-05\\
				1024	9.87865886386179e-06\\
				2048	2.71347371275297e-06\\
				4096	7.54074982760637e-07\\
				8192	2.13669729864059e-07\\
				16384	6.0070244689273e-08\\
				32768	1.65029649618247e-08\\
				65536	4.58542657148393e-09\\
			};
			\addplot [color=darkgray, dashed, line width=0.9pt]
			table[row sep=crcr]{%
				64	0.00346378828198726\\
				128	0.00105243868416859\\
				256	0.000319773350378977\\
				512	9.71600504150757e-05\\
				1024	2.95211448529786e-05\\
				2048	8.9697153275182e-06\\
				4096	2.72536154872724e-06\\
				8192	8.28074838506167e-07\\
				16384	2.51602558378812e-07\\
				32768	7.6447012321933e-08\\
				65536	2.32276878685424e-08\\
			};
			\addlegendentry{\scriptsize $\mathcal{O}(N^{-1.72})$}
			
			\addplot [color=mycolor6, line width=0.9pt, mark=o, mark options={solid, mycolor6}, forget plot]
			table[row sep=crcr]{%
				64	1.28697975158607e-05\\
				128	1.74188571381488e-06\\
				256	3.66679710174919e-07\\
				512	9.70115733301211e-08\\
				1024	8.92657933377613e-09\\
				2048	3.2712163062905e-09\\
				4096	1.59477617206956e-10\\
				8192	2.44761243620679e-11\\
				16384	6.5810458046318e-12\\
				32768	2.58160714759e-12\\
				65536	1.86828486963937e-13\\
			};
			\addplot [color=mycolor7, line width=0.9pt, mark=triangle, mark options={solid, mycolor7}, forget plot]
			table[row sep=crcr]{%
				64	9.63120981937429e-06\\
				128	1.3243470327549e-06\\
				256	1.77417418283065e-07\\
				512	2.43736888765063e-08\\
				1024	3.40110026357102e-09\\
				2048	4.47401526916463e-10\\
				4096	6.06306431225999e-11\\
				8192	8.47943711288475e-12\\
				16384	1.16290107298636e-12\\
				32768	1.56993422619375e-13\\
				65536	2.21601217849075e-14\\
			};
			\addplot [color=black, dashed, line width=0.9pt]
			table[row sep=crcr]{%
				64	2.57395950317214e-05\\
				128	4.23437042835346e-06\\
				256	6.96587996136595e-07\\
				512	1.14594328619067e-07\\
				1024	1.88516888382896e-08\\
				2048	3.10125445419784e-09\\
				4096	5.10181303764536e-10\\
				8192	8.39289282949877e-11\\
				16384	1.38069838168672e-11\\
				32768	2.27135989928539e-12\\
				65536	3.73656973927874e-13\\
			};
			\addlegendentry{\scriptsize $\mathcal{O}(N^{-2.6})$}
			
			\end{axis}
		\end{tikzpicture}
		\caption{Weight sequence $\bsgamma=(\gamma_j)_{j=1}^d$ with $\gamma_j = 1/j^2$.}    
	\end{subfigure}
	\begin{subfigure}[b]{0.5\textwidth}  
		\centering 
		\begin{tikzpicture}
			\begin{axis}[%
			width=0.8\textwidth,
			height=0.8\textwidth,
			at={(0\textwidth,0\textwidth)},
			scale only axis,
			xmode=log,
			xmin=42.6666666666667,
			xmax=98304,
			xminorticks=true,
			xlabel style={font=\color{white!15!black}},
			xlabel={Number of points $N=2^m$},
			ymode=log,
			ymin=1e-15,
			ymax=1,
			yminorticks=true,
			ylabel style={font=\color{white!15!black}},
			ylabel={Worst-case error $e_{N,d,\alpha,\mathbf{\gamma^{\alpha}}}(\bsg)$},
			axis background/.style={fill=white},
			axis x line*=bottom,
			axis y line*=left,
			xmajorgrids,
			xminorgrids,
			ymajorgrids,
			yminorgrids,
			minor grid style={opacity=0},
			legend style={at={(0.03,0.03)}, anchor=south west, legend cell align=left, align=left, draw=white!15!black}
			]
			\addplot [color=mycolor1, line width=0.9pt, mark=o, mark options={solid, mycolor1}, forget plot]
			table[row sep=crcr]{%
				64	0.0131178336032584\\
				128	0.00491437900626413\\
				256	0.00194519098231873\\
				512	0.000757426493570501\\
				1024	0.000276393073757065\\
				2048	0.000106368717079493\\
				4096	3.69171449170946e-05\\
				8192	1.44511369578044e-05\\
				16384	5.85279710148979e-06\\
				32768	2.4613672733137e-06\\
				65536	8.22403863440284e-07\\
			};
			\addplot [color=mycolor2, line width=0.9pt, mark=triangle, mark options={solid, mycolor2}, forget plot]
			table[row sep=crcr]{%
				64	0.0126883996208652\\
				128	0.00477596267110885\\
				256	0.00179072880659425\\
				512	0.000667928409719759\\
				1024	0.000249528553805854\\
				2048	9.3065957832874e-05\\
				4096	3.45161015000271e-05\\
				8192	1.28671774179377e-05\\
				16384	4.77544778112109e-06\\
				32768	1.76952645794544e-06\\
				65536	6.55128760870105e-07\\
			};
			\addplot [color=mycolor3, dashed, line width=0.9pt]
			table[row sep=crcr]{%
				64	0.0262356672065167\\
				128	0.00996814658375531\\
				256	0.00378736113448457\\
				512	0.00143899412418154\\
				1024	0.00054674059745053\\
				2048	0.000207732106669012\\
				4096	7.89270603689721e-05\\
				8192	2.99880502748333e-05\\
				16384	1.13938509185812e-05\\
				32768	4.32905232467889e-06\\
				65536	1.64480772688057e-06\\
			};
			\addlegendentry{\scriptsize $\mathcal{O}(N^{-1.4})$}
			
			\addplot [color=mycolor4, line width=0.9pt, mark=o, mark options={solid, mycolor4}, forget plot]
			table[row sep=crcr]{%
				64	0.000712462383360234\\
				128	0.000184831946815374\\
				256	5.35579779770074e-05\\
				512	1.57046632787247e-05\\
				1024	3.73387111880909e-06\\
				2048	1.07497298654939e-06\\
				4096	2.29271186127242e-07\\
				8192	6.7039957830092e-08\\
				16384	2.16976361387627e-08\\
				32768	8.18186574673242e-09\\
				65536	1.51773888340811e-09\\
			};
			\addplot [color=mycolor5, line width=0.9pt, mark=triangle, mark options={solid, mycolor5}, forget plot]
			table[row sep=crcr]{%
				64	0.000670708492223081\\
				128	0.000175125591950862\\
				256	4.54652696243151e-05\\
				512	1.18025820939106e-05\\
				1024	3.06345545838972e-06\\
				2048	7.94986360971004e-07\\
				4096	2.05111977025723e-07\\
				8192	5.33005202542887e-08\\
				16384	1.376958062069e-08\\
				32768	3.56403343791969e-09\\
				65536	9.2060582961137e-10\\
			};
			\addplot [color=darkgray, dashed, line width=0.9pt]
			table[row sep=crcr]{%
				64	0.00142492476672047\\
				128	0.00038604279131196\\
				256	0.000104587301873436\\
				512	2.83349513560168e-05\\
				1024	7.67654824215092e-06\\
				2048	2.07974215920286e-06\\
				4096	5.63446918110405e-07\\
				8192	1.52649898509437e-07\\
				16384	4.13561433490269e-08\\
				32768	1.12042694388004e-08\\
				65536	3.03547776681623e-09\\
			};
			\addlegendentry{\scriptsize $\mathcal{O}(N^{-1.88})$}
			
			\addplot [color=mycolor6, line width=0.9pt, mark=o, mark options={solid, mycolor6}, forget plot]
			table[row sep=crcr]{%
				64	5.92493821050277e-06\\
				128	7.49632168235576e-07\\
				256	1.07863795593144e-07\\
				512	1.8789107659036e-08\\
				1024	1.81540878550344e-09\\
				2048	3.09512021283855e-10\\
				4096	2.58832445037591e-11\\
				8192	3.87306385928902e-12\\
				16384	8.28402037064509e-13\\
				32768	2.89342697876167e-13\\
				65536	1.4113316510823e-14\\
			};
			\addplot [color=mycolor7, line width=0.9pt, mark=triangle, mark options={solid, mycolor7}, forget plot]
			table[row sep=crcr]{%
				64	5.54168171547323e-06\\
				128	7.02954709134525e-07\\
				256	8.89967547930034e-08\\
				512	1.12675174304235e-08\\
				1024	1.42670185953979e-09\\
				2048	1.80786104448765e-10\\
				4096	2.28473605289084e-11\\
				8192	2.8964272716977e-12\\
				16384	3.66433575093592e-13\\
				32768	4.64189489401149e-14\\
				65536	5.87953689698115e-15\\
			};
			\addplot [color=black, dashed, line width=0.9pt]
			table[row sep=crcr]{%
				64	1.18498764210055e-05\\
				128	1.62707790207184e-06\\
				256	2.2341013571396e-07\\
				512	3.06759059760903e-08\\
				1024	4.21203453660106e-09\\
				2048	5.7834428594703e-10\\
				4096	7.94110566238359e-11\\
				8192	1.09037403279396e-11\\
				16384	1.49716624099736e-12\\
				32768	2.05572279398341e-13\\
				65536	2.8226633021646e-14\\
			};
			\addlegendentry{\scriptsize $\mathcal{O}(N^{-2.86})$}
			\end{axis}
		\end{tikzpicture}
		\caption{Weight sequence $\bsgamma=(\gamma_j)_{j=1}^d$ with $\gamma_j = 1/j^3$.}
	\end{subfigure}
	\vskip\baselineskip
	\hspace{-0.25cm}
	\centering
	\begin{subfigure}[b]{0.5\textwidth}
		\centering
		\begin{tikzpicture}
			\begin{axis}[%
			width=0.8\textwidth,
			height=0.8\textwidth,
			at={(0\textwidth,0\textwidth)},
			scale only axis,
			xmode=log,
			xmin=42.6666666666667,
			xmax=98304,
			xminorticks=true,
			xlabel style={font=\color{white!15!black}},
			xlabel={Number of points $N=2^m$},
			ymode=log,
			ymin=0.000001,
			ymax=1000000000000,
			yminorticks=true,
			ylabel style={font=\color{white!15!black}},
			ylabel={Worst-case error $e_{N,d,\alpha,\mathbf{\gamma^{\alpha}}}(\bsg)$},
			axis background/.style={fill=white},
			axis x line*=bottom,
			axis y line*=left,
			xmajorgrids,
			xminorgrids,
			ymajorgrids,
			yminorgrids,
			minor grid style={opacity=0},
			legend style={at={(0.03,0.03)}, anchor=south west, legend cell align=left, align=left, draw=white!15!black}
			]
			\addplot [color=mycolor1, line width=0.9pt, mark=o, mark options={solid, mycolor1}, forget plot]
			table[row sep=crcr]{%
				64	7243051451.11146\\
				128	3621525725.15113\\
				256	1810762862.20832\\
				512	905381430.73359\\
				1024	452690715.081421\\
				2048	226345357.269875\\
				4096	113172678.465645\\
				8192	56586339.032267\\
				16384	28293169.4120152\\
				32768	14146584.5780807\\
				65536	7073292.22615669\\
			};
			\addplot [color=mycolor2, line width=0.9pt, mark=triangle, mark options={solid, mycolor2}, forget plot]
			table[row sep=crcr]{%
				64	7243051451.09558\\
				128	3621525725.08\\
				256	1810762862.09309\\
				512	905381430.631196\\
				1024	452690714.936109\\
				2048	226345357.135407\\
				4096	113172678.269114\\
				8192	56586338.8874833\\
				16384	28293169.2380297\\
				32768	14146584.4572523\\
				65536	7073292.09608587\\
			};
			\addplot [color=mycolor3, dashed, line width=0.9pt]
			table[row sep=crcr]{%
				64	14486102902.2229\\
				128	7243051429.95876\\
				256	3621525704.40303\\
				512	1810762846.91334\\
				1024	905381420.812582\\
				2048	452690709.084247\\
				4096	226345353.881102\\
				8192	113172676.61004\\
				16384	56586338.1397645\\
				32768	28293168.9872545\\
				65536	14146584.4523134\\
			};
			\addlegendentry{\scriptsize $\mathcal{O}(N^{-1})$}
			
			\addplot [color=mycolor4, line width=0.9pt, mark=o, mark options={solid, mycolor4}, forget plot]
			table[row sep=crcr]{%
				64	10959.825529302\\
				128	5479.69603909232\\
				256	2739.64859547817\\
				512	1369.60526438866\\
				1024	684.715475922968\\
				2048	342.227485853939\\
				4096	171.080148439262\\
				8192	85.4672433901463\\
				16384	42.724700020511\\
				32768	21.3278592624591\\
				65536	10.6585951036708\\
			};
			\addplot [color=mycolor5, line width=0.9pt, mark=triangle, mark options={solid, mycolor5}, forget plot]
			table[row sep=crcr]{%
				64	10959.7656361317\\
				128	5479.51083503195\\
				256	2739.44210618333\\
				512	1369.45795946705\\
				1024	684.527632517571\\
				2048	342.109645550625\\
				4096	170.941903843961\\
				8192	85.3932036462252\\
				16384	42.6417266330489\\
				32768	21.285663384164\\
				65536	10.6181010963113\\
			};
			\addplot [color=darkgray, dashed, line width=0.9pt]
			table[row sep=crcr]{%
				64	21919.6510586041\\
				128	10955.2746207525\\
				256	5475.3628009509\\
				512	2736.54461799129\\
				1024	1367.70415376978\\
				2048	683.568117231064\\
				4096	341.642137743683\\
				8192	170.750137901208\\
				16384	85.339618191816\\
				32768	42.6520910767205\\
				65536	21.3171902073417\\
			};
			\addlegendentry{\scriptsize $\mathcal{O}(N^{-1})$}
			
			\addplot [color=mycolor6, line width=0.9pt, mark=o, mark options={solid, mycolor6}, forget plot]
			table[row sep=crcr]{%
				64	8.57791666330145\\
				128	4.28103520423578\\
				256	2.06796576730862\\
				512	0.971054173490731\\
				1024	0.496436300834007\\
				2048	0.215202611343453\\
				4096	0.106730887187267\\
				8192	0.0459859657856498\\
				16384	0.0248841835926798\\
				32768	0.00934014463870056\\
				65536	0.00501384329233759\\
			};
			\addplot [color=mycolor7, line width=0.9pt, mark=triangle, mark options={solid, mycolor7}, forget plot]
			table[row sep=crcr]{%
				64	8.46766996208136\\
				128	4.0806523512221\\
				256	1.93692177499604\\
				512	0.906951394555028\\
				1024	0.418189445550568\\
				2048	0.190921292344443\\
				4096	0.0859625367018254\\
				8192	0.0379177033253264\\
				16384	0.0166350232569472\\
				32768	0.00714293789454527\\
				65536	0.00304175692103127\\
			};
			\addplot [color=black, dashed, line width=0.9pt]
			table[row sep=crcr]{%
				64	17.1558333266029\\
				128	8.1487447010377\\
				256	3.87052257611543\\
				512	1.83843592624904\\
				1024	0.873227474703245\\
				2048	0.414768995584408\\
				4096	0.197008597051487\\
				8192	0.0935759126776304\\
				16384	0.0444470523850443\\
				32768	0.0211116344921435\\
				65536	0.0100276865846752\\
			};
			\addlegendentry{\scriptsize $\mathcal{O}(N^{-1.07})$}
			\end{axis}
		\end{tikzpicture}
		\caption{Weight sequence $\bsgamma=(\gamma_j)_{j=1}^d$ with $\gamma_j = (0.95)^j$.}
	\end{subfigure}
	\begin{subfigure}[b]{0.5\textwidth}  
		\centering 
		\begin{tikzpicture}
			\begin{axis}[%
			width=0.8\textwidth,
			height=0.8\textwidth,
			at={(0\textwidth,0\textwidth)},
			scale only axis,
			xmode=log,
			xmin=42.6666666666667,
			xmax=98304,
			xminorticks=true,
			xlabel style={font=\color{white!15!black}},
			xlabel={Number of points $N=2^m$},
			ymode=log,
			ymin=1e-12,
			ymax=1,
			yminorticks=true,
			ylabel style={font=\color{white!15!black}},
			ylabel={Worst-case error $e_{N,d,\alpha,\mathbf{\gamma^{\alpha}}}(\bsg)$},
			axis background/.style={fill=white},
			axis x line*=bottom,
			axis y line*=left,
			xmajorgrids,
			xminorgrids,
			ymajorgrids,
			yminorgrids,
			minor grid style={opacity=0},
			legend style={at={(0.03,0.03)}, anchor=south west, legend cell align=left, align=left, draw=white!15!black}
			]
			\addplot [color=mycolor1, line width=0.9pt, mark=o, mark options={solid, mycolor1}, forget plot]
			table[row sep=crcr]{%
				64	0.282174633321939\\
				128	0.13284536582465\\
				256	0.0593909612017073\\
				512	0.0276701354810547\\
				1024	0.0124506955709893\\
				2048	0.0059881638624573\\
				4096	0.002623626543244\\
				8192	0.0011690456214415\\
				16384	0.000510186962064498\\
				32768	0.000247867654332359\\
				65536	0.000101159613448046\\
			};
			\addplot [color=mycolor2, line width=0.9pt, mark=triangle, mark options={solid, mycolor2}, forget plot]
			table[row sep=crcr]{%
				64	0.274712481779285\\
				128	0.124417585308346\\
				256	0.0559772905760882\\
				512	0.0254974002775106\\
				1024	0.0113237967059648\\
				2048	0.00501661471403071\\
				4096	0.00221532998966387\\
				8192	0.000973192798708207\\
				16384	0.000429486397615613\\
				32768	0.000187353189500543\\
				65536	8.07839905190598e-05\\
			};
			\addplot [color=mycolor3, dashed, line width=0.9pt]
			table[row sep=crcr]{%
				64	0.564349266643877\\
				128	0.255268298867404\\
				256	0.115463788575768\\
				512	0.0522269570151156\\
				1024	0.0236234673459449\\
				2048	0.010685443712208\\
				4096	0.00483327471173977\\
				8192	0.00218620256381623\\
				16384	0.000988870266038787\\
				32768	0.000447289020349818\\
				65536	0.000202319226896092\\
			};
			\addlegendentry{\scriptsize $\mathcal{O}(N^{-1.14})$}
			
			\addplot [color=mycolor4, line width=0.9pt, mark=o, mark options={solid, mycolor4}, forget plot]
			table[row sep=crcr]{%
				64	0.0111700804239152\\
				128	0.00471207225113345\\
				256	0.00141560088221411\\
				512	0.000570796290539375\\
				1024	0.000199757941505261\\
				2048	9.25230177509226e-05\\
				4096	2.72459541360662e-05\\
				8192	9.44914547577809e-06\\
				16384	2.89641573116343e-06\\
				32768	1.46751113529886e-06\\
				65536	3.82284119367565e-07\\
			};
			\addplot [color=mycolor5, line width=0.9pt, mark=triangle, mark options={solid, mycolor5}, forget plot]
			table[row sep=crcr]{%
				64	0.00945313308880435\\
				128	0.00316951406710338\\
				256	0.00109688680255328\\
				512	0.000381105276844875\\
				1024	0.000123143143042282\\
				2048	4.29822557594017e-05\\
				4096	1.38276588312289e-05\\
				8192	4.80757770222531e-06\\
				16384	1.54392319759943e-06\\
				32768	5.05565921043526e-07\\
				65536	1.61957971127494e-07\\
			};
			\addplot [color=darkgray, dashed, line width=0.9pt]
			table[row sep=crcr]{%
				64	0.0223401608478305\\
				128	0.0079894943028627\\
				256	0.00285727661722161\\
				512	0.00102184560847563\\
				1024	0.000365441848110695\\
				2048	0.000130692683163539\\
				4096	4.67395223639288e-05\\
				8192	1.67154189349267e-05\\
				16384	5.97792223879755e-06\\
				32768	2.13787966859995e-06\\
				65536	7.64568238735131e-07\\
			};
			\addlegendentry{\scriptsize $\mathcal{O}(N^{-1.48})$}
			
			\addplot [color=mycolor6, line width=0.9pt, mark=o, mark options={solid, mycolor6}, forget plot]
			table[row sep=crcr]{%
				64	0.000140326391649271\\
				128	5.22382192248969e-05\\
				256	4.94191767928417e-06\\
				512	1.57769490883062e-06\\
				1024	3.63557481789908e-07\\
				2048	1.53192720884102e-07\\
				4096	1.80126021224901e-08\\
				8192	3.44699584558495e-09\\
				16384	4.71780877474591e-10\\
				32768	2.8928012044642e-10\\
				65536	3.60374117983146e-11\\
			};
			\addplot [color=mycolor7, line width=0.9pt, mark=triangle, mark options={solid, mycolor7}, forget plot]
			table[row sep=crcr]{%
				64	6.5512951579925e-05\\
				128	1.08514549896987e-05\\
				256	2.07633554526966e-06\\
				512	4.1358368190811e-07\\
				1024	8.0791551612054e-08\\
				2048	1.40823660331657e-08\\
				4096	2.45320848306968e-09\\
				8192	4.39452153240671e-10\\
				16384	8.4772239825475e-11\\
				32768	1.53147764523675e-11\\
				65536	2.64712334202646e-12\\
			};
			\addplot [color=black, dashed, line width=0.9pt]
			table[row sep=crcr]{%
				64	0.000280652783298542\\
				128	6.15362145887965e-05\\
				256	1.34924929708976e-05\\
				512	2.95837772580286e-06\\
				1024	6.48656907763745e-07\\
				2048	1.42225173046636e-07\\
				4096	3.1184436033961e-08\\
				8192	6.83753114814163e-09\\
				16384	1.49920403084707e-09\\
				32768	3.28716999953848e-10\\
				65536	7.20748235966293e-11\\
			};
			\addlegendentry{\scriptsize $\mathcal{O}(N^{-2.19})$}
		\end{axis}
		\end{tikzpicture}
		\caption{Weight sequence $\bsgamma=(\gamma_j)_{j=1}^d$ with $\gamma_j = (0.7)^j$.}
	\end{subfigure}
	\vskip\baselineskip
	\begin{tikzpicture}
		\hspace{0.05\linewidth}
		\begin{customlegend}[
		legend columns=5,legend style={align=left,draw=none,column sep=1.5ex},
		legend entries={CBC-DBD, standard fast CBC \quad, $\alpha=1.5$, $\alpha=2$, $\alpha=3$}
		]
		\addlegendimage{color=gray, mark=o,solid,line width=1.0pt,line legend}
		\addlegendimage{color=gray, mark=triangle,solid,line width=1.0pt}  
		\addlegendimage{color=mycolor-alpha1, only marks,mark=square*, solid, line width=5pt}
		\addlegendimage{color=mycolor-alpha2, only marks,mark=square*, solid, line width=5pt}
		\addlegendimage{color=mycolor-alpha3, only marks,mark=square*, solid, line width=5pt}
		\end{customlegend}
	\end{tikzpicture}
	\caption{Convergence results of the worst-case error $e_{2^m,d,\alpha,\bsgamma^\alpha}(\bsg)$ in the weighted space $W_{d,\bsgamma^\alpha}^\alpha$
		for smoothness parameters $\alpha=1.5,2,3$ with dimension $d=100$. The generating vectors $\bsg$ are constructed via the component-by-component digit-by-digit algorithm and the standard CBC construction for polynomial lattice rules for $N=2^m$, respectively.}  
	\label{fig:cbc-dbd}
\end{figure}

\newpage

\subsection{Computational complexity}

We demonstrate the computational complexity of Algorithm \ref{alg:fast-cbcdbd} which was proved in Theorem~\ref{thm:cost-cbcdbd}. For this purpose, we measure and compare the computation times of implementations of Algorithm~\ref{alg:fast-cbcdbd} and the standard fast CBC algorithm for polynomial lattice rules with primitive modulus $p \in \F_2[x]$, cf., e.g., \cite{NC06a}. For all timings we perform three independent measurements and then select the lowest time out of these three runs. We consider multiple values of $m,d \in \N$ and fix the positive weight sequence $\bsgamma = (\gamma_j)_{j \ge 1}$ with $\gamma_j = 1 / j^2$. Note that the chosen weight sequence does not affect the computation times.

In Table \ref{tab:cbcdbd_times} we display the timing results for the two considered algorithms. Furthermore, Figure \ref{fig:times_cbc-dbd} provides a graphical illustration of the running times of both algorithms. We remark that the measured times only indicate the duration for the construction of the generating vectors but do not include the calculation of the corresponding worst-case error. All timings were performed on an Intel Core i5 CPU with 2.3 GHz using Python 3.6.3.

\begin{table}[H]
	\captionof{table}{Computation times (in seconds) for constructing the generating vector $\bsg$ of a polynomial lattice rule with $2^m$ points in $d$ dimensions using the component-by-component digit-by-digit algorithm (\textbf{bold font}) and the standard fast CBC construction (normal font). For the CBC algorithm we constructed the polynomial lattice rules with smoothness parameter $\alpha=2$.}
	\label{tab:cbcdbd_times}
	\centering
	\begin{tabular}{p{2cm}p{2.1cm}p{2.1cm}p{2.1cm}p{2.1cm}p{2.1cm}} 
		\toprule[1.2pt]
		& $d=50$ & $d=200$ & $d=500$ & $d=1000$ & $d=2000$ \\ 
		\toprule[1.2pt]
		\multirow{2}{6em}{$m=10$}
		& 0.007 & 0.025 & 0.061 & 0.12 & 0.239 \\
		& \textbf{0.068} & \textbf{0.268} & \textbf{0.67} & \textbf{1.338} & \textbf{2.682} \\
		\midrule
		\multirow{2}{6em}{$m=12$}
		& 0.025 & 0.089 & 0.213 & 0.421 & 0.827 \\
		& \textbf{0.107} & \textbf{0.433} & \textbf{1.082} & \textbf{2.175} & \textbf{4.318} \\
		\midrule
		\multirow{2}{6em}{$m=14$}
		& 0.117 & 0.399 & 0.953 & 1.839 & 3.763 \\
		& \textbf{0.203} & \textbf{0.816} & \textbf{2.037} & \textbf{4.077} & \textbf{8.147} \\
		\midrule
		\multirow{2}{6em}{$m=16$}
		& 0.586 & 2.0 & 4.804 & 9.523 & 18.836 \\
		& \textbf{0.573} & \textbf{2.31} & \textbf{5.82} & \textbf{11.606} & \textbf{23.083} \\
		\midrule
		\multirow{2}{6em}{$m=18$}
		& 2.858 & 9.466 & 22.715 & 44.56 & 88.198 \\
		& \textbf{2.556} & \textbf{10.36} & \textbf{26.019} & \textbf{51.599} & \textbf{103.685} \\
		\midrule
		\multirow{2}{6em}{$m=20$}
		& 13.703 & 44.914 & 106.861 & 211.073 & 416.24 \\
		& \textbf{16.812} & \textbf{67.824} & \textbf{169.935} & \textbf{340.589} & \textbf{687.135} \\
		\midrule
	\end{tabular}
\end{table}

The timings displayed in Table \ref{tab:cbcdbd_times} and Figure \ref{fig:times_cbc-dbd} confirm that the computational complexity of both algorithms depends on $m$ and $d$ in a similar way and the measured times are in accordance with Proposition \ref{thm:cost-cbcdbd}. Additionally, the linear dependence of the construction cost on the dimension $d$ is well observable. The measured construction times for Algorithm \ref{alg:fast-cbcdbd} are slightly higher than for the fast CBC algorithm but in general both algorithms can be executed in comparable time. This is especially remarkable since the fast CBC construction is based on fast Fourier transformations which rely on compiled and optimized code via Python's Discrete Fourier Transform (numpy.fft) library while the CBC-DBD construction does not make use of any compiled libraries. Lastly, we remark that the slight parabola shape of the timing curve of the CBC-DBD algorithm in Figure~\ref{fig:times_cbc-dbd}, which one might suspect, is not to be observed for larger values of $m$.
  
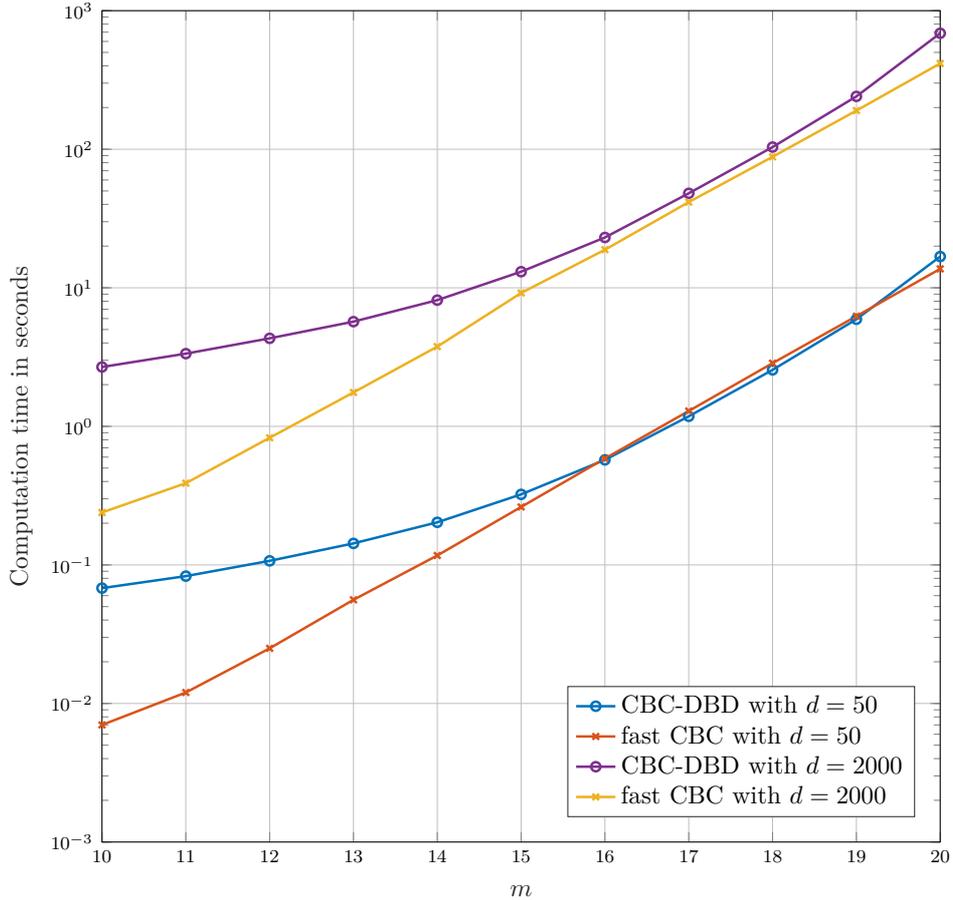
\begin{figure}[H]
	\centering
	\textbf{Computation times for CBC-DBD and fast CBC algorithm.} \par\medskip 
	\hspace{-0.25cm}
	\resizebox{.8\linewidth}{!}{
	\begin{tikzpicture}
		\begin{axis}[%
		width=0.8\textwidth,
		height=0.8\textwidth,
		at={(0\textwidth,0\textwidth)},
		scale only axis,
		xmin=10,
		xmax=20,
		xlabel style={font=\color{white!15!black}},
		xlabel={$m$},
		ymode=log,
		ymin=0.001,
		ymax=1000,
		yminorticks=true,
		ylabel style={font=\color{white!15!black}},
		ylabel={Computation time in seconds},
		axis background/.style={fill=white},
		xmajorgrids,
		ymajorgrids,
		yminorgrids,
		minor grid style={opacity=0},
		legend style={at={(0.97,0.03)}, anchor=south east, legend cell align=left, align=left, draw=white!15!black}
		]
		\addplot [color=mycolor1-time, line width=1.1pt, mark=o, mark options={solid, mycolor1-time}]
		table[row sep=crcr]{%
			10	0.068\\
			11	0.083\\
			12	0.107\\
			13	0.143\\
			14	0.203\\
			15	0.323\\
			16	0.573\\
			17	1.182\\
			18	2.556\\
			19	5.932\\
			20	16.812\\
		};
		\addlegendentry{CBC-DBD with $d=50$}
		
		\addplot [color=mycolor2-time, line width=1.1pt, mark=x, mark options={solid, mycolor2-time}]
		table[row sep=crcr]{%
			10	0.007\\
			11	0.012\\
			12	0.025\\
			13	0.056\\
			14	0.117\\
			15	0.262\\
			16	0.586\\
			17	1.291\\
			18	2.858\\
			19	6.239\\
			20	13.703\\
		};
		\addlegendentry{fast CBC with $d=50$}
		
		\addplot [color=mycolor3-time, line width=1.1pt, mark=o, mark options={solid, mycolor3-time}]
		table[row sep=crcr]{%
			10	2.682\\
			11	3.348\\
			12	4.318\\
			13	5.699\\
			14	8.147\\
			15	13.069\\
			16	23.083\\
			17	48.079\\
			18	103.685\\
			19	240.728\\
			20	687.135\\
		};
		\addlegendentry{CBC-DBD with $d=2000$}
		
		\addplot [color=mycolor4-time, line width=1.1pt, mark=x, mark options={solid, mycolor4-time}]
		table[row sep=crcr]{%
			10	0.239\\
			11	0.389\\
			12	0.827\\
			13	1.758\\
			14	3.763\\
			15	9.173\\
			16	18.836\\
			17	41.551\\
			18	88.198\\
			19	190.144\\
			20	416.24\\
		};
		\addlegendentry{fast CBC with $d=2000$}
		\end{axis}
		\end{tikzpicture}
	}
	\caption{Computation times (in seconds) for constructing the generating vector $\bsg$ of a polynomial lattice rule with $2^m$ points in $d \in \{50,2000\}$ dimensions using the component-by-component digit-by-digit algorithm (circles) and the standard fast CBC construction (crosses).}  
	\label{fig:times_cbc-dbd}
\end{figure}

\section{Conclusion}

In this paper, we presented an algorithm for constructing good polynomial lattice rules for numerical integration in weighted Walsh spaces. In particular, we studied a component-by-component digit-by-digit (CBC-DBD) construction with quality measure independent of the smoothness parameter $\alpha$, similar to \cite{EKNO2020}, where such an algorithm was analyzed for ordinary lattice rules. The construction algorithm is formulated for the special case of product weights and yields polynomial lattice rules which admit error convergence rates that are arbitrarily close to the optimal convergence order. Furthermore, the proven error bounds become independent of the dimension if the weights satisfy suitable summability conditions. In addition to these theoretical results, we derived a fast implementation of the considered algorithm which exhibits the same computational complexity as the state-of-the-art fast CBC algorithm, but does not rely on the use of fast Fourier transformations (FFTs). The considered algorithm is, to the best of our knowledge, the first construction method for good polynomial lattice rules with modulus $p(x)=x^m$ that requires only 
$\calO(d \,m\, 2^m)$ operations. Extensive numerical experiments illustrated our findings and proved that the considered method is competitive with the standard fast CBC algorithm.

\bigskip

\bigskip

\begin{small}
\noindent\textbf{Authors' addresses:}\\
	
	\noindent Adrian Ebert\\
	Johann Radon Institute for Computational and Applied Mathematics (RICAM) \\
	Austrian Academy of Sciences\\
	Altenbergerstr. 69, 4040 Linz, Austria.\\
	\texttt{adrian.ebert@oeaw.ac.at}
	
	\medskip
	
	\noindent Peter Kritzer\\
	Johann Radon Institute for Computational and Applied Mathematics (RICAM) \\
	Austrian Academy of Sciences\\
	Altenbergerstr. 69, 4040 Linz, Austria.\\
	\texttt{peter.kritzer@oeaw.ac.at}
	
	\medskip
	
	\noindent Onyekachi Osisiogu\\
	Johann Radon Institute for Computational and Applied Mathematics (RICAM)\\
	Austrian Academy of Sciences\\
	Altenbergerstr. 69, 4040 Linz, Austria.\\
	\texttt{onyekachi.osisiogu@oeaw.ac.at}
	
	\medskip
	
	\noindent Tetiana Stepaniuk\\
	Institute of Mathematics\\
	University of L\"ubeck\\
	Ratzeburger Allee 160, 23562 L\"ubeck, Germany,\\
	\texttt{stepaniuk@math.uni-luebeck.de}
	
\end{small}
\end{document}